\documentclass[leqno,12pt]{amsart}
\usepackage{amssymb}
\usepackage{amsmath}
\usepackage{enumerate}
\usepackage{amsfonts}
\usepackage{hyperref}

\usepackage[headheight=18pt, top=20mm, bottom=20mm, left=22mm, right=22mm]{geometry}

\usepackage{tikz}
\usepackage{pgfplots}
\pgfplotsset{compat=1.18}

\definecolor{darkgreen}{RGB}{45, 119, 75}

\newcommand{\supp}{\text{supp }}

\newtheorem{theorem}{Theorem}[section]
\newtheorem{corollary}[theorem]{Corollary}
\newtheorem{lemma}[theorem]{Lemma}
\newtheorem{proposition}[theorem]{Proposition}
\newtheorem{remark}[theorem]{Remark}
\newtheorem{definition}[theorem]{Definition}

\numberwithin{equation}{section}

\hypersetup{
    colorlinks=true,
    linkcolor= blue,
    citecolor =cyan,
    urlcolor = teal,
}

\makeatletter
\@ifundefined{c@part}{\newcounter{part}}{}

\renewcommand\part{%
  \@ifstar{\@spart}{\@part}%
}

\def\@part#1{%
  \refstepcounter{part}%
  \addcontentsline{toc}{part}{\protect\numberline{\thepart}#1}%
  \par\vspace{1\bigskipamount}
  \begin{center}%
    \normalfont\large\bfseries Part\ \thepart\quad #1%
  \end{center}%
  \vspace{1\bigskipamount}%
  \@afterheading
}

\def\@spart#1{%
  \addcontentsline{toc}{part}{#1}%
  \par\vspace{1\bigskipamount}
  \begin{center}%
    \normalfont\large\bfseries #1%
  \end{center}%
  \vspace{1\bigskipamount}%
  \@afterheading
}
\makeatother
  
\begin{document}

\title[H\"ormander's multiplier theorem on Hardy spaces]{H\"ormander's multiplier theorem on $H^p$-spaces in the rational Dunkl setting}


\author[Jacek Dziubański]{Jacek Dziubański}
\author[Agnieszka Hejna-Łyżwa]{Agnieszka Hejna-Łyżwa}
\begin{abstract} 
On $\mathbb{R}^N$ equipped with a normalized root system $\mathcal R$ and a multiplicity function $k\geq 0$, let $dw(\mathbf x)=\Pi_{\alpha\in \mathcal R}|\langle \mathbf x,\alpha\rangle|^{k(\alpha)}\, d\mathbf x$, $\mathbf{N}=N+\sum_{\alpha\in \mathcal R}k(\alpha)$ denote the associated measure and the homogeneous dimension of the system $(\mathcal R,k)$ respectively. Let $\mathcal F$ stand for the Dunkl transform. For $0<p\leq 1$, let $m$ be a bounded function on $\mathbb{R}^N$, which satisfies the classical H\"ormander's  condition with smoothness $s>\mathbf{N}/p$. We show that the multiplier operator $\mathcal T_mf=\mathcal F^{-1}(m\mathcal Ff)$, initially defined on $H^p_{\mathrm{Dunkl}}\cap L^2(dw)$, has a unique extension to a bounded operator in $H^p_{\mathrm{Dunkl}}$, where the space $H^p_{\mathrm{Dunkl}}$ is defined by means of a Littlewood-Paley square function. To prove the theorem, we use special atomic and molecule characterizations of $H^p_{\mathrm{Dunkl}}$. 

\end{abstract}

\address{Jacek Dziubański, Uniwersytet Wroc\l awski,
Instytut Matematyczny,
Pl. Grunwaldzki 2,
50-384 Wroc\l aw,
Poland}
\email{jdziuban@math.uni.wroc.pl}

\address{Agnieszka Hejna-Łyżwa, Uniwersytet Wroc\l awski,
Instytut Matematyczny,
Pl. Grunwaldzki 2,
50-384 Wroc\l aw,
Poland}
\email{hejna@math.uni.wroc.pl}

\maketitle

\section{Introduction}

{In classical harmonic analysis, H\"ormander's multiplier theorem \cite{Hormander1960} reduces the problem of operator boundedness on the spaces $L^p(\mathbb{R}^N)$ and $H^p(\mathbb{R}^N)$ to controlling the behavior of the symbol on dyadic scales. By replacing strict pointwise bounds on derivatives with a Sobolev-type integral condition, this theorem establishes the smoothness threshold $s > N/p - N/2$ required for the multiplier (see e.g.,~\cite{Calderon-Torchinsky},~\cite{Taibleson-Weiss}). This result serves as a standard tool in operator theory, finding direct applications in, among others, proving the boundedness of singular integral operators, investigating the properties of Bochner-Riesz means, and analyzing solutions to partial differential equations (e.g., the wave and Schr\"odinger equations). The aim of the present paper is to investigate an analogous sufficient condition for multiplier operators on the Hardy spaces $H^p_{\mathrm{Dunkl}}$ in the rational Dunkl setting.}

On the Euclidean space $\mathbb{R}^N$ equipped with a root system $\mathcal R$ and a multiplicity function $k\geq 0$, \textit{} let 
$$dw(\mathbf x)=\Pi_{\alpha\in \mathcal R}|\langle \mathbf x,\alpha\rangle|^{k(\alpha)}\, d\mathbf x$$ 
be the associated measure and $\mathbf{N}=N+\sum_{\alpha\in \mathcal{R}}k(\alpha)$ the homogeneous dimension of the system $(\mathcal R,k)$, respectively. Let $E(\mathbf x,\mathbf y)$, $\mathbf x,\mathbf y\in\mathbb C^N$, be the associated Dunkl kernel. The Dunkl transform 
$$ \mathcal Ff(\xi)={\mathbf c}_k^{-1} \int_{\mathbb{R}^N} E(-i\xi,\mathbf x)f(\mathbf x)\, dw(\mathbf x),$$
originally defined on $L^1(dw)$, is an isometry on $L^2(dw)$ and preserves the Schwartz class of functions $\mathcal S(\mathbb{R}^N)$. 
Moreover,  its inverse is given by $\mathcal F^{-1} f(\mathbf x)=\mathcal Ff(-\mathbf x)$. 

Let 
$\Delta_k=\sum_{j=1}^N T_j$
be the Laplace-Dunkl operator, where $T_j$ are the Dunkl operators. The operator $\Delta_k$ generates the Dunkl-heat semigroup $ \{H_t\}_{t>0}$, of   contractions on  $L^p(dw)$, $1\leq p\leq \infty$, strongly continuous and holomorphic for $1\leq p<\infty$. 

For a reasonable function $f$ on $\mathbb{R}^N$, (for example for $f\in L^2(dw)$) we define the  square function $Sf(x)$: 
\begin{equation}\label{eq:square_conic}
    Sf(\mathbf{x})=\Big(\int_0^\infty \int_{\|\mathbf{x}-\mathbf{y}\|<t} |t^2\Delta_kH_{t^2}f(\mathbf{y})|^2\frac{dw(\mathbf{y})}{w(B(\mathbf{x},t))}\frac{ dt }{t}\Big)^{1/2}.
\end{equation}
Following \cite{Hofman} and \cite{Duong}, for $0<p\leq 1$, let 
\begin{equation}\label{eq:HHp} \mathbb{H}^p_{\mathrm{Dunkl}} =\{f\in L^2(dw): Sf\in L^p(dw)\}.\end{equation}
We equip the space with the quasi-norm $\|f\|_{\mathbb{H}^p_{\mathrm{Dunkl}}}=\|Sf\|_{L^p(dw)}$. Then 
$\|f-g\|_{\mathbb{H}^p_{\mathrm{Dunkl}}}^p$ is a metric in $\mathbb{H}^p_{\mathrm{Dunkl}}$.

\begin{definition}
    The Hardy space $H^p_{\mathrm{Dunkl}}$ is defined as the  completion of the space $\mathbb{H}^p_{\mathrm{Dunkl}}$ in the quasi-norm $\|\cdot\|_{\mathbb{H}^p_{\mathrm{Dunkl}}}$, { where, for a Cauchy sequence $(f_n)_{n \in \mathbb{N}}$ in $\mathbb H^p_{\rm Dunkl}$, its equivalence class $\big[(f_n)_{n \in \mathbb{N}}\big]$ has the quasi-norm $\|\big[(f_n)_{n \in \mathbb{N}}\big]\|_{H^p_{\rm Dunkl}}=\lim_{n\to \infty } \| f_n\|_{\mathbb H^p_{\rm Dunkl}}$.} 
\end{definition}

{
Let us note that if $(f_n)_{n \in \mathbb{N}}$ is a Cauchy sequence in the quasi-norm $\|\cdot\|_{\mathbb{H}^p_{\mathrm{Dunkl}}}$, then the limit $\lim_{n \to \infty}\|f_n\|_{\mathbb{H}^p_{\mathrm{Dunkl}}}$ exists, since for all $n,m \in \mathbb{N}$ we have
\[\Big|\|f_n\|_{\mathbb{H}^p_{\mathrm{Dunkl}}}-\|f_m\|_{\mathbb{H}^p_{\mathrm{Dunkl}}}\Big| \leq \|f_n-f_m\|_{\mathbb{H}^p_{\mathrm{Dunkl}}}.\]}

We will show that every element of $H^p_{\mathrm{Dunkl}}$ is identified with a unique tempered distribution $f\in\mathcal S'(\mathbb{R}^N)$, where $\mathcal S(\mathbb{R}^N)$ denotes the class of Schwartz functions on $\mathbb{R}^N$. 

For $s>0$, denote by $W_2^s(\mathbb{R}^N)$ the classical Sobolev space, that is, 
$$ W_2^s(\mathbb{R}^N)=\Big\{f\in L^2(dx): \int_{\mathbb{R}^N} |\widehat f(\xi)|^2 (1+|\xi|)^{2s}\, d\xi  =\|f\|_{W_2^s}^2<\infty\Big\}.$$

We are now in a position to state our main result.

\begin{theorem}\label{teo:teo_main}
Let $0<p \leq 1$ and let $\psi \in C^{\infty}_c(\mathbb{R}^N)$ be a non-zero radial function such that $\supp \psi \subset \mathbb{R}^N\setminus\{0\}$. Assume that  $m$ is a bounded function on $\mathbb{R}^N$ which satisfies the H\"ormander condition 
\begin{equation}\label{eq:assumption}
M_0:=\sup_{t>0}\|\psi (\cdot)m(t \cdot )\|_{W^{s}_2} < \infty
\end{equation}
for certain  $s> {\mathbf{N}}/p$, then the multiplier operator $\mathcal{T}_mf=\mathcal{F}^{-1}(m(\cdot)\mathcal{F}f(\cdot))$
originally defined on $\mathbb{H}^p_{\mathrm{Dunkl}}$, has a unique extension to a bounded operator on $H^p_{\mathrm{ Dunkl}}$.
Moreover, 
$$ \| \mathcal T_m f\|_{H^p_{\mathrm{Dunkl}}}\leq C_{p,s} M_0 \|f\|_{H^p_{\mathrm{Dunkl}}}.$$
\end{theorem}

The theorem extends the results of the authors of~\cite{DzHe-JFA}, where Dunkl multipliers on Lebesgue spaces $L^p(dw)$, $1\leq p<\infty$, and the Hardy space $H^1_{\rm Dunkl}$  were studied. In the case of the space $H^1_{\rm Dunkl}$, the multiplier theorem was proved using a  characterization of this space by relevant Riesz transforms. We refer the reader to \cite{Thangavelu2}, in which the Dunkl-transform  multiplier  operators on Lebesgue spaces $L^{p}_r(dw)$, $1<p\leq 2$,  of radial functions were considered.

Let us also remark that for the H\"ormander's multiplier theorem on the classical Hardy spaces one needs smoothness $s>N/p-N/2$ for the multiplier $m$, where $N$ is the Euclidean dimension  which, in this case, coincides with the homogeneous dimension. 
{It should be emphasized that Theorem~\ref{teo:teo_main} applies to general, non-radial multipliers. The requirement of a higher smoothness order $s > \mathbf{N}/p$ in our proof, compared to the classical setting, is a consequence of a fundamental obstacle in Dunkl theory: the lack of general $L^p(dw)$-boundedness for Dunkl translations and convolutions. Whether the Dunkl translation operators are bounded on $L^p(dw)$ for $p \neq 2$ remains a well-known open problem. Currently, affirmative results are restricted either to product root systems or to the subspace of radial functions. 

To overcome this structural difficulty, our strategy avoids general $L^p$ estimates. Instead, we exploit two crucial features of the Dunkl setting. First, we rely on the fact that Dunkl translations and convolutions are well-behaved bounded operators on $L^2(dw)$. Second, we make essential use of the support theorem, which states that if a function $f$ is supported in a ball $B(0,r)$, then its Dunkl translation $\tau_{\mathbf{x}} f$ is supported in the orbit set $\mathcal{O}(B(\mathbf{x},r))$ (see Theorem~\ref{teo:support}). This result was previously known only for radial functions \cite{Roesle99}, but in \cite{DzHe-JFA} it was generalized to non-radial functions, which serves as one of our main tools (see also~\cite{AAS-Colloq},~\cite{ThangaveluXu}). Furthermore, recent results from~\cite{GallardoRejeb2018},~\cite{JiuLi}, and~\cite{DH-heat} indicate that this orbital support condition cannot be essentially improved. By combining the $L^2$-theory with the support estimate, we are able to effectively employ atomic decompositions and Hardy space techniques to obtain the desired bounds.

One of the main elements of the proof of Theorem~\ref{teo:teo_main} is the usage of characterizations of the space $H^p_{\mathrm{Dunkl}}$ by means of special atoms and molecules. In the case of Hardy spaces associated with semigroups   satisfying Gaussian bounds, or more generally Davies-Gaffney estimates, such characterizations were considered in \cite{Hofman}, \cite{Duong}. 
Since in the estimates for the Dunkl-heat semigroups kernels two distances occur, namely the Euclidean distance and an orbit distance, and for the reader convenience, we provide the details.}

The paper is organized as follows. Section~\ref{sec:preliminaries} contains necessary preliminaries regarding Dunkl theory. In Section~\ref{sec:tent}, we recall the Calder\'on reproducing formulas and properties of tent spaces. Section~\ref{sec:Hardy} is devoted to the atomic and molecular characterizations of the Hardy spaces $H^p_{\mathrm{Dunkl}}$. Finally, the proof of Theorem~\ref{teo:teo_main} is presented in Section~\ref{sec:Hormander}.

\section{Preliminaries - Dunkl theory}\label{sec:preliminaries}

In this section, we present basic facts concerning the theory of the Dunkl operators.   For more details, we refer the reader to~\cite{Dunkl},~\cite{Roesle99},~\cite{Roesler3}, and~\cite{Roesler-Voit}. 

We consider the Euclidean space $\mathbb{R}^N$ with the scalar product $\langle \mathbf{x},\mathbf y\rangle=\sum_{j=1}^N x_jy_j
$, where $\mathbf x=(x_1,...,x_N)$, $\mathbf y=(y_1,...,y_N)$, and the norm $\| \mathbf x\|^2=\langle \mathbf x,\mathbf x\rangle$.

A {\it normalized root system}  in $\mathbb{R}^N$ is a finite set  $\mathcal R\subset \mathbb{R}^N\setminus\{0\}$ such that $\mathcal{R} \cap \alpha  \mathbb{R} = \{\pm \alpha\}$,  $\sigma_\alpha (\mathcal{R})=\mathcal{R}$, and $\|\alpha\|=\sqrt{2}$ for all $\alpha\in \mathcal{R}$, where $\sigma_\alpha$ is defined by $\sigma_\alpha (\mathbf x)=\mathbf x-2\frac{\langle \mathbf x,\alpha\rangle}{\|\alpha\|^2} \alpha$.
The finite group $G$ generated by the reflections $\sigma_{\alpha}$, $\alpha \in \mathcal{R}$, is called the {\it reflection group} of the root system. A~{\textit{multiplicity function}} is a $G$-invariant function $k:\mathcal{R}\to\mathbb C$, which will be fixed and non-negative throughout this paper.  

The associated $G$-invariant measure $dw$ is defined by $dw(\mathbf x)=w(\mathbf x)\, d\mathbf x$, where 
 \begin{equation}\label{eq:measure}
w(\mathbf x)=\prod_{\alpha\in \mathcal{R}}|\langle \mathbf x,\alpha\rangle|^{k(\alpha)}.
\end{equation}
Let $\mathbf{N}=N+\sum_{\alpha \in \mathcal{R}}k(\alpha)$. Then, 
\begin{equation}\label{eq:t_ball} w(B(t\mathbf x, tr))=t^{\mathbf{N}}w(B(\mathbf x,r)) \ \ \text{ for all } \mathbf x\in\mathbb{R}^N, \ t,r>0.
\end{equation}
Observe that there is a constant $C>1$ such that for all $\mathbf{x} \in \mathbb{R}^N$ and $r>0$, we have
\begin{equation}\label{eq:balls_asymp}
C^{-1}w(B(\mathbf x,r))\leq  r^{N}\prod_{\alpha \in \mathcal{R}} (|\langle \mathbf x,\alpha\rangle |+r)^{k(\alpha)}\leq C w(B(\mathbf x,r)),
\end{equation}
so $dw(\mathbf x)$ is doubling.

Moreover, by \eqref{eq:balls_asymp}, there exists a constant $C \geq 1$ such that for every $\mathbf{x} \in \mathbb{R}^N$,
\begin{equation}\label{eq:growth}
C^{-1} \left(\frac{R}{r}\right)^{ N}\leq \frac{w(B(\mathbf x, R))}{w(B(\mathbf x, r))}\leq C \left(\frac{R}{r}\right)^{\mathbf{N}}\ \ \text{ for } 0<r<R.
\end{equation}

For $\mathbf{x},\mathbf{y} \in \mathbb{R}^N$, let 
\begin{equation}\label{eq:distance_of_orbits}
    d(\mathbf x,\mathbf y)=\min_{\sigma\in G}\| \sigma(\mathbf x)-\mathbf y\|.    
\end{equation}

For a Lebesgue measurable set $A$ (in particular, for $A=B(\mathbf{x}_0,r)$) we denote
\begin{equation}
    \mathcal{O}(A)=\{\sigma(\mathbf{z})\;:\; \sigma \in G,\, \mathbf{z} \in A\}.
\end{equation}

For $\xi \in \mathbb{R}^N$, the {\it Dunkl operators} $T_\xi$  are the following $k$-deformations of the directional derivatives $\partial_\xi$ by   difference operators:
\begin{equation}\label{eq:T_xi}
     T_\xi f(\mathbf x)= \partial_\xi f(\mathbf x) + \sum_{\alpha\in \mathcal{R}} \frac{k(\alpha)}{2}\langle\alpha ,\xi\rangle\frac{f(\mathbf x)-f(\sigma_\alpha(\mathbf{x}))}{\langle \alpha,\mathbf x\rangle}.
\end{equation}
 We simply write $T_j$, if $\xi=e_j$, {where $\{e_j\}_{j=1}^N$ stands for the canonical basis in $\mathbb{R}^N$.}  

The Dunkl operators $T_{\xi}$, which were introduced in~\cite{Dunkl}, commute with each other and are skew-symmetric with respect to the $G$-invariant measure $dw$.

\subsection{Dunkl kernel}\label{sec:Dunkl_transform}
For any fixed $\mathbf y\in\mathbb{R}^N$, the {\it Dunkl kernel} $\mathbf{x} \longmapsto E(\mathbf x,\mathbf y)$ is a unique analytic solution to the system $T_\xi f=\langle \xi ,\mathbf y\rangle f, \quad f(0)=1$. The function $E(\mathbf x,\mathbf y)$, which generalizes the exponential function $e^{\langle\mathbf x,\mathbf y\rangle}$, has a unique extension to a holomorphic  function $E(\mathbf z,\mathbf w)$ on $\mathbb C^N\times \mathbb C^N$. Moreover, for all $\xi,\mathbf{x} \in \mathbb{R}^N$, we have
\begin{equation} \label{eq:E} | E(i\xi, \mathbf x)|\leq 1.
\end{equation}

For a multi-index $\boldsymbol{\beta}=(\beta_1,\beta_2,\ldots,\beta_N)\in \mathbb N_0^N$, we denote
\begin{equation}\label{eq:iterated_der_ord}
    |\boldsymbol{\beta}|=\beta_1+\ldots +\beta_N, \ \partial^{\mathbf{0}}=I, \ \ \partial^{\boldsymbol{\beta}}=\partial_1^{\beta_1} \circ \ldots \circ \partial_N^{\beta_N}, \ \ T^{\mathbf{0}}=I, \ \ T^{\boldsymbol{\beta}}=T_{1}^{\beta_1} \circ \ldots \circ T_{N}^{\beta_N}.
\end{equation}

\subsection{Dunkl transform}
Let $f \in L^1(dw)$. The \textit{Dunkl transform }$\mathcal{F}f$ of $f$ is defined  by
\begin{equation}\label{eq:Dunkl_transform}
    \mathcal{F} f(\xi)=\mathbf{c}_k^{-1}\int_{\mathbb{R}^N}f(\mathbf{x})E(\mathbf{x},-i\xi)\, {dw}(\mathbf{x}), \text{ where } \mathbf{c}_k=\int_{\mathbb{R}^N}e^{-\frac{{\|}\mathbf{x}{\|}^2}2}\,{dw}(\mathbf{x}){>0}.
\end{equation}
The Dunkl transform is a generalization of the Fourier transform on $\mathbb{R}^N$. It was introduced in~\cite{D5} for $k \geq 0$ and further studied in~\cite{dJ1} in a more general setting. It maps the Schwartz class  $\mathcal S(\mathbb{R}^N)$ onto $\mathcal S(\mathbb{R}^N)$ and possesses many properties analogous to those of the classical Fourier transform, for example, 
\begin{equation}\label{eq:T_j_transform}
    \mathcal{F}(T_{j}f)(\xi)=i\xi_{j}\mathcal{F}f(\xi) \text{ for all }f \in \mathcal{S}(\mathbb{R}^N) \text{ and }j \in \{1,\ldots,N\}, 
\end{equation}
\begin{equation}
   \text{ if } f \ \text{ is real-valued and radial, then } \mathcal Ff(\xi) \ \text{ is real-valued and radial}.
\end{equation}
Moreover, it was proved in~\cite[Corollary 2.7]{D5} (see also~\cite[Theorem 4.26]{dJ1}) that $\mathcal F$ extends uniquely  to  an isometry on $L^2(dw)$, i.e.,
   \begin{equation}\label{eq:Plancherel}
       \|f\|_{L^2(dw)}=\|\mathcal{F}f\|_{L^2(dw)} \text{ for all }f \in L^2(dw),
   \end{equation}
and preserves the Schwartz class of functions $\mathcal S(\mathbb{R}^N)$ (\cite{D5}, see also \cite{dJ1}).
Furthermore,  the following inversion formula holds (\cite[Theorem 4.20]{dJ1}): for all $f \in L^1(dw)$ such that $\mathcal{F}f \in L^1(dw)$ one has  
$f(\mathbf{x})=(\mathcal{F})^2f(-\mathbf{x}) \text{ for almost all }\mathbf{x} \in \mathbb{R}^N$.

\subsection{Dunkl translations}
Suppose that $f \in \mathcal{S}(\mathbb{R}^N)$.  The \textit{Dunkl translation }$\tau_{\mathbf{x}}f$ of $f$ is defined by
\begin{equation}\label{eq:translation}
    \tau_{\mathbf{x}} f(-\mathbf{y})=\mathbf{c}_k^{-1} \int_{\mathbb{R}^N}{E}(i\xi,\mathbf{x})\,{E}(-i\xi,\mathbf{y})\,\mathcal{F}f(\xi)\,{dw}(\xi)=\mathcal{F}^{-1}(E(i\,\cdot\,,\mathbf{x})\mathcal{F}f(\cdot) )(-\mathbf{y}).
\end{equation}
The Dunkl translation was introduced in~\cite{R1998}. The definition can be extended to functions which are not necessarily in $\mathcal{S}(\mathbb{R}^N)$. For instance, using Plancherel's theorem, one can define the Dunkl translation of an $L^2(dw)$ function $f$ by
\begin{equation}\label{eq:translation_Fourier}
    \tau_{\mathbf{x}}f(-\mathbf{y})=\mathcal{F}^{-1}(E(i\, \cdot\, ,\mathbf{x})\mathcal{F}f(\cdot))(-\mathbf{y})
\end{equation}
(see~\cite{R1998} and~\cite[Definition 3.1]{ThangaveluXu}). In particular, the operators $f \mapsto \tau_{\mathbf{x}}f$ are contractions on $L^2(dw)$. Here and henceforth, for a reasonable function $g(\mathbf{x})$, we  write $g(\mathbf x,\mathbf y):=\tau_{\mathbf x}g(-\mathbf y)$. 

 The following specific formula  for the Dunkl translations of (reasonable) radial functions $f({\mathbf{x}})=\widetilde{f}({\|\mathbf{x}\|})$ was obtained by R\"osler \cite{Roesler2003}:
\begin{equation}\label{eq:translation-radial}
\tau_{\mathbf{x}}f(-\mathbf{y})=\int_{\mathbb{R}^N}{(\widetilde{f}\circ A)}(\mathbf{x},\mathbf{y},\eta)\,d\mu_{\mathbf{x}}(\eta)\text{ for }\mathbf{x},\mathbf{y}\in\mathbb{R}^N.
\end{equation}
Here
\begin{equation*}
A(\mathbf{x},\mathbf{y},\eta)=\sqrt{{\|}\mathbf{x}{\|}^2+{\|}\mathbf{y}{\|}^2-2\langle \mathbf{y},\eta\rangle}=\sqrt{{\|}\mathbf{x}{\|}^2-{\|}\eta{\|}^2+{\|}\mathbf{y}-\eta{\|}^2}
\end{equation*}
and $\mu_{\mathbf{x}}$ is a probability measure, 
which is supported in the set $\operatorname{conv}\mathcal{O}(\mathbf{x})$,  where $\mathcal O(\mathbf x) =\{\sigma(\mathbf x): \sigma \in G\}$ is the orbit of $\mathbf x$.

\subsection{Dunkl convolution}
Assume that $f,g \in L^2(dw)$. The \textit{generalized convolution} (or the \textit{Dunkl convolution}) $f*g$ is defined by the formula
\begin{equation}\label{eq:conv_transform}
    f*g(\mathbf{x})=\mathbf{c}_k\mathcal{F}^{-1}\big((\mathcal{F}f)(\mathcal{F}g)\big)(\mathbf{x}),
\end{equation}
and equivalently, by
\begin{equation}\label{eq:conv_translation}
    (f*g)(\mathbf{x})=\int_{\mathbb{R}^N}f(\mathbf{y})\,\tau_{\mathbf{x}}g(-\mathbf{y})\,{dw}(\mathbf{y})=\int_{\mathbb{R}^N}g(\mathbf{y})\,\tau_{\mathbf{x}}f(-\mathbf{y})\,{dw}(\mathbf{y}).
\end{equation}
The generalized convolution of $f,g \in \mathcal{S}(\mathbb{R}^N)$ was considered in~\cite{R1998} and~\cite{Trimeche}, the definition was extended to wider classes of functions (see e.g.~\cite{ThangaveluXu}). 

In the proof of Theorem~\ref{teo:teo_main}, we will need the following lemma, {which is a consequence of \eqref{eq:Plancherel}} (see~\cite[page 8]{ThangaveluXu}).

\begin{lemma}\label{lem:ThangaveluXu}
    Let $f \in L^2(dw) \cap L^1(dw)$ and $g \in L^2(dw)$. Then
    \begin{equation*}
        \|f*g\|_{L^2(dw)} \leq \|f\|_{L^1(dw)}\|g\|_{L^2(dw)}.
    \end{equation*}
\end{lemma}

For $f \in \mathcal{S}(\mathbb{R}^N)$, $\mathbf{x},\mathbf{y} \in \mathbb{R}^N$ let us denote
\begin{equation*}
    f(\mathbf x,\mathbf y):=\tau_{\mathbf x}f(-\mathbf y).
\end{equation*}

The result concerning the support of generalized translation of compactly supported function was proved in~\cite{DzHe-JFA}.

\begin{theorem}\label{teo:support}
Let $f \in L^2(dw)$, $\supp\, f \subseteq B(0,r)$, and $\mathbf{x} \in \mathbb{R}^N$. Then  
 $$f(\mathbf x,\mathbf y)=\tau_{\mathbf x}f(-\mathbf y)=0 \quad \text{if \ \ } d(\mathbf x,\mathbf y)>r. $$
\end{theorem}

\subsection{Dunkl Laplacian, Dunkl heat semigroup and Dunkl heat kernel}
The \textit{Dunkl Laplacian} associated with $\mathcal{R}$ and $k$  is the differential-difference operator 
\begin{equation}\label{eq:Dunkl_laplacian_form}
\Delta_k=\sum_{j=1}^N T_{j}^2.    
\end{equation}
It was introduced in~\cite{Dunkl}, where it was also proved that $\Delta_k$ acts on $C^2(\mathbb{R}^N)$ functions by
\begin{equation}\label{eq:laplace_formula}
    \Delta_k f(\mathbf x)=\Delta f(\mathbf x)+\sum_{\alpha\in \mathcal{R}} k(\alpha) \delta_\alpha f(\mathbf x), \text{ where }\delta_\alpha f(\mathbf x)=\frac{\partial_\alpha f(\mathbf x)}{\langle \alpha , \mathbf x\rangle} -  \frac{f(\mathbf x)-f(\sigma_\alpha (\mathbf x))}{\langle \alpha, \mathbf x\rangle^2}.
\end{equation}
Here $\Delta=\sum_{j=1}^{N}\partial_j^2$. It follows from~\eqref{eq:T_j_transform} that for all $\xi \in \mathbb{R}^N$ and $f \in \mathcal{S}(\mathbb{R}^N)$, we have
\begin{equation}\label{eq:Laplacian_on_Fourier_side}
    \mathcal{F}(\Delta_{k}f)(\xi)=-\|\xi\|^2\mathcal{F}f(\xi).
\end{equation} 
The operator $(-\Delta_{k},\mathcal{S}(\mathbb{R}^N))$ in $L^2(dw)$ is densely defined and closable. It is essentially self-adjoint on $L^2(dw)$ (see, for instance, \cite[Theorem\;3.1]{AH}).

The closure of $(-\Delta_{k},\mathcal{S}(\mathbb{R}^N))$ generates a strongly continuous and positivity-preserving contraction semigroup on $L^2(dw)$,  which is given by
 \begin{equation}\label{eq:heat_semigroup}
  H_t f(\mathbf x)=f*h_t(\mathbf x)=\int_{\mathbb{R}^N} h_t(\mathbf x,\mathbf y)f(\mathbf y)\, dw(\mathbf y), 
  \end{equation}
where 
 \begin{equation}\label{eq:heat_kernel} \ h_t(\mathbf x,\mathbf y)=h_t(\mathbf y,\mathbf x)={\mathbf c}_k^{-1} (2t)^{-\mathbf{N}/ 2}
 e^{-(\|\mathbf x\|^2+\|
 \mathbf y\|^2)\slash (4t)}E\left(\frac{\mathbf x}{\sqrt{2t}},\frac{\mathbf y}{\sqrt{2t}}\right)
 \end{equation}
 is the so-called \textit{generalized heat kernel} (or the \textit{Dunkl heat kernel}), see~\cite{R1998} .  The integral kernels $h_t(\mathbf x,\mathbf y)$ are the generalized translations of the Schwartz-class functions:
 $$ h_t(\mathbf x,\mathbf y)=\tau_{\mathbf x}h_t(-\mathbf y), \quad h_t(\mathbf x)={\mathbf c}_k^{-1} (2t)^{-\mathbf{N}/2} e^{-\|\mathbf x\|^2/4t}. $$
 Moreover, for all $t>0$ and $\mathbf{x},\mathbf{y} \in \mathbb{R}^N$,  one has $h_t(\mathbf{x},\mathbf{y})>0$, and $\int_{\mathbb{R}^N}h_t(\mathbf{x},\mathbf{z})\,dw(\mathbf{z})=1$.
 Formula~\eqref{eq:heat_semigroup} defines  contraction semigroups on the $L^p(dw)$-spaces, $1\leq p\leq  \infty$, which are  strongly continuous for $1\leq p<\infty$. 
  The semigroup $\{H_t\}_{t>0}$ on $L^2(dw)$ can be expressed by means of the Dunkl transform, that is,  
\begin{equation}\label{eq:heat_transform}
    \mathcal{F}(H_tf)(\xi){=\mathcal F(h_t*f)(\xi)=\mathbf c_k \mathcal F(h_t)(\xi)\mathcal Ff(\xi)}=e^{-t\|\xi\|^2}\mathcal{F}f(\xi),\quad f\in L^2(dw).
\end{equation} 
 Note that in the case $k \equiv 0$ the Dunkl heat kernel is the classical heat kernel.

 \subsection{Estimates for generalized translations of some functions}
 
 In order to prove an atomic characterization of Hardy space, we need  bounds for the Dunkl heat kernel and its derivatives. 
 
 For $\mathbf{x},\mathbf{y} \in \mathbb{R}^N$ and $t,r>0$, we denote
\begin{equation}\label{eq:V}
    V(\mathbf{x},\mathbf{y},r):=\max\{w(B(\mathbf{x},r)),w(B(\mathbf{y},r))\}, \ \ \mathcal G_t(\mathbf x,\mathbf y)=\frac{1}{V(\mathbf x,\mathbf y,\sqrt{t})}e^{-\frac {d(\mathbf x,\mathbf y)^2}{t}}.
\end{equation} 

The following theorem was proved in~\cite{DH-atom}, see also~\cite[Theorem 4.1]{ADzH}.
For more detailed upper and lower bounds for $h_t(\mathbf x,\mathbf y)$ we refer the reader to~\cite{DH-heat}.

\begin{theorem}[Theorem 4.1, \texorpdfstring{~\cite{DH-atom}}{[DH]}]\label{teo:heat_new}   For every nonnegative integer $m$ and for all multi-indices $\boldsymbol{\alpha},\boldsymbol{\beta} \in \mathbb{N}_0^N$ there are constants $C_{m,\boldsymbol{\alpha},\boldsymbol{\beta}}, c>0$ such that
  \begin{equation}\label{eq:heat2} |\partial_t^m \partial_{\mathbf x}^{\boldsymbol{\alpha}}\partial_{\mathbf y}^{\boldsymbol{\beta}} h_t(\mathbf{x},\mathbf{y})|
  \leq C_{m,\boldsymbol{\alpha},\boldsymbol{\beta}} t^{-m-\frac{|\boldsymbol{\alpha}|}{2}-\frac{|\boldsymbol{\beta}|}{2}} \Big(1+\frac{\| \mathbf x-\mathbf y\|}{\sqrt{t}}\Big)^{-2} \mathcal G_{t\slash c} (\mathbf x,\mathbf y).
  \end{equation}
\end{theorem}

 We have the following estimate for generalized translations of the  Schwartz class functions.
 
\begin{theorem}{\texorpdfstring{\cite[Theorem 4.1 and Remark 4.2]{ADzH}}{\cite[Theorem 4.1 and Remark 4.2]{ADzH}}}\label{teo:translacja}
     Let $\varphi \in \mathcal{S}(\mathbb{R}^N)$ and $M>0$. For $t>0$ set  $  \varphi_t(\mathbf x)=t^{-\mathbf{N}}\varphi(\mathbf x/t)$. There is  $C>0$ such that for all $\mathbf{x},\mathbf{y} \in \mathbb{R}^N$ and $t>0$, we have 
   \begin{equation}\label{eq:g_t}
       |\varphi_t(\mathbf{x},\mathbf{y})| \leq C\left(1+\frac{\|\mathbf{x}-\mathbf{y}\|}{t}\right)^{-1}\Big(1+\frac{d(\mathbf x,\mathbf y)}{t}\Big)^{-M} \frac{1}{w(B(\mathbf{x},t))}.
   \end{equation}
\end{theorem}

For a positive integer $M$ we denote the domain of the operator $\Delta_k$ by

\begin{equation}\label{eq:domain}
    \mathcal{D}(\Delta_k^M)=\{f\in L^2(dw): \int_{\mathbb{R}^N} |\mathcal Ff(\xi)|^2\|\xi\|^{4M}\, dw(\xi)<\infty\}.
\end{equation}

We say that a tempered distribution $F$ coincides with a $dw(\mathbf x)$-locally integrable function $
f$ if 
$$ \langle F,\varphi\rangle= \int_{\mathbb{R}^N} f(\mathbf x)\varphi(\mathbf x)\, dw(\mathbf x) \quad \text{ for all } \varphi\in\mathcal S(\mathbb{R}^N).$$ 
Then we simply write $F=f$. 

{
\subsection{Operator $t^2\Delta_kH_{t^2}$}\label{sec:Qt}

Let $Q(\mathbf x)=\Delta_k h_1(\mathbf x)$, $Q_t(\mathbf x)=t^{-\mathbf{N}}Q(\mathbf x/t)$,  $Q_t(\mathbf x,\mathbf y)=\tau_{\mathbf x}Q_t(-\mathbf y) $. Then, 
\begin{equation}
    Q_t(\mathbf x)=t^2\Delta_k h_{t^2}(\mathbf x), \quad  Q_t(\mathbf x,\mathbf y)=t^2\Delta_{k\, {\mathbf x}} h_{t^2}(\mathbf x,\mathbf y) =t^2\Delta_{k\, {\mathbf y}} h_{t^2}(\mathbf x,\mathbf y)
\end{equation}
are the convolution and integral kernels of the operators $t^2\Delta_kH_{t^2}$.  

Since $\frac{d^n}{dt^n}h_t(\mathbf x)=\Delta_k^nh_t(\mathbf x)$,  the function $\mathbf y\mapsto Q_t(\mathbf x,\mathbf y)$ belongs to the Schwartz class $\mathcal S(\mathbb{R}^N)$, for all $\mathbf x$ (see e.g., Theorem \ref{teo:heat_new}).
 Thus  the convolution $f*Q_t(\mathbf x)=\langle f,Q_t(\mathbf x, \cdot)\rangle$ is well-defined for any tempered distribution $f\in\mathcal S'(\mathbb{R}^N)$ and  $Q_t*f(\mathbf x)$ is a $C^\infty$-function of $\mathbf x$ and $t$ on $\mathbb{R}^N\times (0,\infty)$.
  
  Furthermore, from Theorem  \ref{teo:heat_new} we easily conclude that 

\begin{equation}\label{eq:Qt-bound} \begin{split}|Q_t(\mathbf x,\mathbf y)| & \leq  \frac{C'}{w(B(\mathbf x,t))}\exp\Big(-c\frac{d(\mathbf x,\mathbf y)^2}{t^2}\Big)\\
& \leq  \frac{C}{w(B(\mathbf x,t))}\exp\Big(-c\frac{d(\mathbf x,\mathbf y)}{t}\Big), \quad \mathbf x,\mathbf y\in\mathbb R^N, \ t>0.
\end{split}\end{equation}
 }

 \begin{lemma}\label{lem:Q_operator_bounded}
    For $t>0$ and $M_1 \in \mathbb{N}$, and  $f\in L^2(dw)$, {let  $\mathcal{Q}_t\,f= t^{2M_1}\Delta_k^{M_1}Q_t*f$.} 
    There is a constant $C>0$ such that for all $t>0$ we have
    \begin{equation*}
        \|\mathcal{Q}_t\|_{L^2(dw)\to L^2(dw)}\leq C.
    \end{equation*}
\end{lemma}

 \begin{proof}
     By \eqref{eq:heat_transform},~\eqref{eq:Laplacian_on_Fourier_side},  $\mathcal F(\mathcal Q_t f)(\xi)=(-1)^{M_1+1}t^{2{M_1}+2} \|\xi\|^{2M_1+2} e^{-t^2\|\xi\|^2}\mathcal Ff(\xi)$. Hence,  the lemma is a consequence of Plancherel's equality~\eqref{eq:Plancherel} and the fact that 
     \begin{equation*}
       \sup_{t>0, \xi\in\mathbb R^N}   t^{2{M_1}+2} \|\xi\|^{2M_1+2} e^{-t^2\|\xi\|^2} < \infty .
     \end{equation*}
 \end{proof}

\section{Preliminaries - Calder\'on reproducing formulas and tent spaces}\label{sec:tent}

\subsection{Calder\'on reproducing formula.}\label{subsection-Cal}
Fix an {$m\in\mathbb{N}$} sufficiently large.
Let $\widetilde\Theta\in C^m(\mathbb{R})$ be an even function such that $$\sum_{j=0}^m (1+s^2)^m\Big|\frac{d^j}{ds^j}\widetilde \Theta(s)\Big|<\infty.$$ 
Set $\Theta(\mathbf{x})=\widetilde\Theta (\|\mathbf{x}\|)$ and assume that $\int_{\mathbb{R}^N}\Theta(\mathbf{x})\, {dw}(\mathbf{x})=0$. Write $L^2\Big(\mathbb{R}^{1+N}_+,\frac{dt}{t}\,dw(\mathbf{x})\Big)=L^2(dw\,dt/t)$. The Plancherel theorem for the Dunkl transform implies {that}
\begin{equation}\label{L2toL2}
\|\Theta_t*f(\mathbf{x})\|_{L^2(dw\,dt/t)}\leq C\,\|f\|_{L^2({dw})}\,.
\end{equation}
{Thus, $f\mapsto \Theta_t*f(\mathbf{x})$ is a bounded linear operator from $L^2(dw)$ into $L^2(dw\,dt/t)$. 
By duality, for $F(t,\mathbf{x})\in L^2(dw\,dt/t)$, the limit 
$$
\lim_{\varepsilon\to 0^{+}}\int_{\varepsilon }^{\varepsilon^{-1}}(\Theta_t*F(t,\cdot))(\mathbf{x})\frac{dt}{t}
$$
exists in $L^2(\mathbb{R}^N,{dw})$ and defines a bounded linear operator {$\pi_{\Theta}$} from $L^2(dw\,dt/t)$ into $L^2(\mathbb{R}^N,dw)$, that is, 
\begin{equation}\label{L2backL2}
\|{\pi_{\Theta}F}\|_{L^2({dw})}\leq C\,\|{F}\|_{L^2(dw\,dt/t)}.
\end{equation}
With the customary abuse of notation, we write 
\begin{equation}\label{eq:reproduction}
    \pi_\Theta F(\mathbf{x})=\int_0^\infty (\Theta_t * F(t,\cdot ))(\mathbf{x})\frac{dt}{t}=\int_0^\infty\int_{\mathbb{R}^N} \Theta_t(\mathbf{x},\mathbf{y})F(t,\mathbf{y})\, {dw}(\mathbf{y})\frac{dt}{t}.
\end{equation} }

Now, let us specify the context in which the operator $ \pi_\Theta$ will be used. In this paper, we will use two variants of Calder\'on reproducing formulas. Let $M$ be a positive integer. Let $\psi$ and $\phi$ be compactly supported  $C^\infty$, real-valued,  radial functions such that 
\begin{equation}\label{eq:psi_support}
    \supp \psi \subseteq B(0,1),
\end{equation}
\begin{equation}\label{eq:phi_form}
   \phi=\Delta_k^{M}\psi,
\end{equation}
\begin{equation}
    \quad |\mathcal F\psi(\xi)| \leq C \|\xi\|^2 \quad \text{ for } \|\xi\|\leq 1,
\end{equation}
\begin{equation}\label{eq:repro_1}
    \mathbf{c}_k^3\int_0^\infty  \mathcal F\phi(t\xi)^2t^2\|\xi\|^2e^{-t^2\|\xi\|^2}\frac{dt}{t}=1,\quad \xi\ne 0,
\end{equation}
and there is a constant $c_{\phi}$ such that
\begin{equation}
    \mathbf{c}_k^2\int_0^\infty  \mathcal F\phi(t\xi)^2\frac{dt}{t}=c_\phi,\quad \xi\ne 0.
\end{equation}
By the Plancherel theorem for the Dunkl transform and the normalization condition~\eqref{eq:repro_1}, applying the definition~\eqref{eq:reproduction} with $\Theta_t=\phi_t$ yields
\begin{equation}\label{eq:reproduction_1}
     f=-\int_0^\infty \phi_t*\phi_t * (t^2\Delta_k h_{t^2}*f)\,\frac{dt}{t}=-\int_0^\infty t^{4M}(\Delta_k^{M}\psi_t)*(\Delta_k^M\psi_t) * (t^2\Delta_k h_{t^2}*f)\,\frac{dt}{t}
\end{equation}
with convergence in $L^2(dw)$. To simplify notation, we write $f*f=f^{*2}$. Similarly, one can note that the mapping
\begin{align*}
    L^2(dw) \ni f \longmapsto \frac{1}{\sqrt{c_{\phi}}}\phi_t*f(\mathbf x) \in L^2(dw(\mathbf x)\,dt/t)
\end{align*}
is an isometric embedding. By composing this operator with its adjoint (which amounts to applying~\eqref{eq:reproduction} with $\Theta_t=\frac{1}{\sqrt{c_\phi}}\phi_t$), we get
\begin{equation}\label{eq:reproduction_2}
     f=\frac{1}{c_\phi}\int_0^\infty \phi_t*\phi_t * f\,\frac{dt}{t}=\frac{1}{c_\phi}\int_0^\infty \phi_t^{*2}* f\,\frac{dt}{t}=\frac{1}{c_\phi}\int_0^\infty t^{4M}(\Delta_k^{M}\psi_t)^{*2}* f\,\frac{dt}{t}
\end{equation}
with convergence in $L^2(dw)$.

We will need this lemma in the proof of the atomic and molecular decompositions of the Hardy space.

\begin{lemma}\label{lem:eta}
    Assume that $\eta \in \mathcal{S}(\mathbb{R}^N)$ and that there is a constant $C>0$ such that for all $\xi \in \mathbb{R}^N$ we have
    $|\mathcal{F}\eta(\xi)| \leq C\|\xi\|^2$. Then, there is a constant $C>0$ such that for all $g \in L^2(dw)$ we have
    \begin{equation*}
        \int_0^{\infty}\int_{\mathbb{R}^N}|\eta_t*g(\mathbf{x})|^2\,dw(\mathbf{x})\,\frac{dt}{t} \leq C\|g\|_{L^2(dw)}^2.
    \end{equation*}
\end{lemma}

\begin{proof}
    By Plancherel's theorem for the Dunkl transform (see~\eqref{eq:Plancherel}) applied to the inner integral (and the formula $\mathcal{F}(f_1*f_2)=\mathbf{c}_k(\mathcal F f_1)(\mathcal{F}f_2)$ for $f_1,f_2 \in L^2(dw)$), we have
    \begin{equation*}
        \int_0^{\infty}\int_{\mathbb{R}^N}|\eta_t*g(\mathbf{x})|^2\,dw(\mathbf{x})\,\frac{dt}{t} =\mathbf{c}_k^{2}\int_0^{\infty}\int_{\mathbb{R}^N}|\mathcal{F}g(\xi)|^2|\mathcal{F}\eta(t\xi)|^2\,dw(\xi)\frac{dt}{t}.
    \end{equation*}
    By the assumption (the estimate $|\mathcal{F}\eta(\xi)| \leq C\|\xi\|^2$) together with the fact that $\mathcal{S}(\mathbb{R}^N)$ is preserved by the Dunkl transform (so $\mathcal{F}\eta \in \mathcal{S}(\mathbb{R}^N)$), it is straightforward to prove that there is a constant $c_{\eta}>0$ such that for all $\xi \in \mathbb{R}^N \setminus \{0\}$ we have $\int_0^{\infty}|\mathcal{F}\eta(t\xi)|^2\,\frac{dt}{t} \leq c_{\eta}$. Therefore, by Tonelli's and  Plancherel's theorems, we arrive at 
    \begin{equation*}
        \mathbf{c}_k^{2}\int_0^{\infty}\int_{\mathbb{R}^N}|\mathcal{F}g(\xi)|^2|\mathcal{F}\eta(t\xi)|^2\,dw(\xi)\,\frac{dt}{t} \leq \mathbf{c}_k^{2}c_\eta\int_{\mathbb{R}^N}|\mathcal{F}g(\xi)|^2\,dw(\xi)=\mathbf{c}_k^2c_\eta\|g\|_{L^2(dw)}^2.
    \end{equation*}
\end{proof}
\subsection{Tent spaces  \texorpdfstring{$T_2^p$}{T2p} on spaces of homogeneous type.}
 The tent spaces on Euclidean spaces were introduced  in \cite{CMS} and then extended to spaces of homogeneous type (see, e.g. \cite{Rus}). For more details, we refer the reader to \cite{St2}. In what follows, we restrict our attention to a specific space of homogeneous type, namely $\mathbb{R}^N$ equipped with the standard Euclidean distance and a doubling measure $dw$.

For a measurable function $F(t, \mathbf{x})$ on $(0,\infty)\times\mathbb{R}^N=:\mathbb R^{1+N}_+$, let
\[\mathcal{A}F(\mathbf{x}) :=\Big( \int_0^\infty\int_{\|\mathbf{y}-\mathbf{x}\|<t} |F(t,\mathbf{y})|^2\frac{{dw}(\mathbf{y})}{{w}(B(\mathbf{x},t))}\frac{dt}{t}\Big)^{1/ 2}.\]

\begin{definition}\normalfont
For $0< p<\infty$, the tent space $T_2^p$ is defined as 
$$T_2^p=\{F: \|F\|_{T_2^p}:=\|\mathcal{A}F\|_{L^p({dw})}<\infty\}. $$
\end{definition}
Clearly, by the doubling property,
\begin{equation*}
    \mathcal{A}F(\mathbf{x}) \sim \Big( \int_0^\infty\int_{\|\mathbf{y}-\mathbf{x}\|<t} |F(t,\mathbf{y})|^2\frac{{dw}(\mathbf{y})}{{w}(B(\mathbf{y},t))}\frac{dt}{t}\Big)^{1/2}.
\end{equation*}
Consequently,
\begin{equation}\label{T22}
\|F\|_{T_2^2}^2=\|\mathcal{A}F\|_{L^2({dw})}^2\sim \int_0^\infty \int_{\mathbb{R}^N} |F(t,\mathbf{y})|^2\frac{{dw}(\mathbf{y})dt}{t}.
\end{equation}

Now, an $L^2(dw)$-function $f$ belongs to $\mathbb{H}^p_{\mathrm{Dunkl}}$ if and only if $F(t,\mathbf x)=Q_t*f(\mathbf x)$ belongs to the tent space $T_2^p$. Moreover, $\|f\|_{H^p_{\mathrm{Dunkl}}}=\|F\|_{T_2^p}$. 

If $\Omega\subset\mathbb{R}^N$ is an open set, then the tent over $\Omega$ is given by
\begin{equation}\label{eq:tent_def}
\widehat\Omega=\Big( (0,\infty) \times \mathbb{R}^N\Big) \setminus \bigcup_{\mathbf{x} \in\Omega^c} \Gamma (\mathbf{x}), \ \ \text{ where } \Gamma (\mathbf{x})=\{ (t, \mathbf{y}): \|\mathbf{x}-\mathbf{y}\|<t\}.
\end{equation}

As in the classical case (see~\cite{CMS}), the tent space $T_2^p$ on the space of homogeneous type   admits the following  atomic decomposition (see{,} e.g.,~\cite{Rus}).

\begin{definition} \label{def:tent_atom_def}\normalfont
Let $0<p\leq 1$. A measurable function $A(t,\mathbf{x})$ is a {\it $T^p_2$-atom} if there is a Euclidean  ball $B\subset\mathbb{R}^N$ such that the following conditions hold:

$\bullet$  $ \text{supp}\, A\subseteq \widehat B$

$\bullet$ $\iint_{(0,\infty) \times \mathbb{R}^N} |A(t,\mathbf{x})|^2\, {dw}(\mathbf{x})\, \frac{dt}{t}\leq {w}(B)^{1-2/p}.$
\end{definition}

\begin{theorem}\label{teo:atomic_decomposition_tent}
    A function $F$ belongs to $T_2^p$ if and only if there are   sequences $\{A_j\}_{j \in \mathbb{N}}$ of $T_2^p$-atoms and $\{\lambda_j\}_{j \in \mathbb{N}}$ of scalars in $\mathbb{C}$ such that
$$  \sum_{j=1}^{\infty} \lambda_j A_j=F,\ \ \ \sum_{j=1}^{\infty} |\lambda_j|^p\sim \|F\|_{T_2^p}^p,$$
where the convergence is in the $T_2^p$ norm and almost everywhere.

\end{theorem}

\begin{remark}\label{remarkAtom}
\normalfont
According to the proof of the atomic decomposition of $T_2^p$ presented in \cite{Rus}, the function $\lambda_jA_j$ can be chosen to be of the form $\lambda_j A_j(t,\mathbf{x}) =\chi_{S_j}(t,\mathbf{x})F(t,\mathbf{x}) $, where the sets $S_j$ are pairwise disjoint, $\mathbb{R}_+^{N+1}=\bigcup S_j$, and each $S_j$ is contained in a tent $\widehat B_j$. So, if $F\in T^p_2\cap T^2_2$, then $F$ can be decomposed into atoms such that
$ F(t,\mathbf{x})=\sum_{j} \lambda_j A_j(t,\mathbf{x})$ and the convergence is in the spaces  $T_2^p$, $T_2^2$, and pointwise.
\end{remark}
Throughout this paper, $C$ denotes a generic positive constant whose value may change from line to line.

\section{Hardy spaces}\label{sec:Hardy}

The Hardy space $H^p_{\mathrm{Dunkl}}$ is defined as the completion of the space $\mathbb  H^p_{\mathrm{Dunkl}}\subset L^2(dw)$ with respect to the quasinorm $\|\cdot\|_{\mathbb{H}^p_{\mathrm{Dunkl}}}$. {It will turn out that if $(f_n)_{n \in \mathbb{N}}$ is a Cauchy sequence in $\mathbb{H}^p_{\mathrm{Dunkl}}$, then it converges in $\mathcal S'(\mathbb{R}^N)$ to a unique tempered distribution $f$ (see Corollary \ref{coro:distribution}). Moreover, the distribution $f$ admits a special  atomic decomposition (see Section \ref{section-atomic} for the definition of the atomic Hardy spaces and Section \ref{sec:atomic_proof} for the proof).} 

\subsection{Atoms and molecules}

We begin by introducing the notion of a $(p,2,M,\Delta_k)$-atom.

\begin{definition} (cf. \cite{Hofman}, \cite{Duong})
    Let $0<p\leq 1$, $M>\mathbf{N}(2-p)/4p$, $M\in\mathbb Z$. A function $\boldsymbol{a}(\cdot)$ is called a \textbf{$(p,2,M,\Delta_k)$-atom} if there exist $\boldsymbol{b}\in\mathcal  D(\Delta_k^M)$ and {a ball $B=B(\mathbf x_0,r_B)$}  such that 

    \begin{equation}\label{eq:cancellation}
        \boldsymbol{a}=\Delta_k^M \boldsymbol{b},
    \end{equation}
    \begin{equation}\label{eq:support}
        \supp \boldsymbol{b}\subseteq \mathcal O(B), 
    \end{equation}
    and
\begin{equation}\label{eq:size}
    \| (r_B^2 \Delta_k)^m \boldsymbol{b}\|_{L^2(dw)}\leq r_B^{2M} w(B)^{\frac{1}{2}-\frac{1}{p}}, \quad m=0,1,\dots, M.
\end{equation}
\end{definition}

The following proposition establishes that every atom is uniformly bounded in $\mathbb{H}^p_{\mathrm{Dunkl}}$.

\begin{proposition}\label{atom_in_Hp}
Suppose $0<p\leq 1$ and  $M>\mathbf{N}(2-p)/4p$.   There is a constant $C>0$ such that if   $\boldsymbol{a}$ is a  $(p,2,M,\Delta_k)$-atom, then $\boldsymbol{a}\in\mathbb{H}^p_{\mathrm{Dunkl}}$ and 
    \[ \|\boldsymbol{a}\|_{\mathbb{H}^p_{\mathrm{Dunkl}}}\leq C. \]
\end{proposition}

Instead of proving Proposition~\ref{atom_in_Hp} directly, we establish a more general result for molecules. 

The concept of molecular characterization for Hardy spaces was originally developed by Taibleson and Weiss~\cite{Taibleson-Weiss} to provide a more flexible alternative to atomic decomposition. Inspired by their seminal work, and adapting these classical ideas to our specific setting associated with the operator $\Delta_k$ and the measure $dw$, in the spirit of~\cite{Hofman} and~\cite{Duong}, the notion of a molecule is formulated as follows.

\begin{definition}\label{def:molecule}
     Fix $\varepsilon>0$ and $0<p\leq 1$,   $M>\mathbf N(2-p)/4p$, $M\in\mathbb Z$. We say that an $L^2(dw)$-function $\boldsymbol{m}$ is a $(p,2,M,\Delta_k,\varepsilon)$\textbf{-molecule} if the following conditions hold:
     \begin{enumerate}[(a)]
         \item{ there exist a function $\boldsymbol{b}\in \mathcal{D}(\Delta_k^M)$ and a ball $B=B(\mathbf{x}_0,r)$ such that $\boldsymbol{m}=\Delta_k^M \boldsymbol{b}$,}\label{numitem:1}
         \item{for all $j=0,1,2,\ldots$ and $n=0,1,\ldots,M$,  we have
         \begin{equation}\label{eq:molecule_condition}
    \|(r^2\Delta_k)^n \boldsymbol{b}\|_{L^2(U_j(B),dw)}\leq r^{2M}2^{-j\varepsilon} w(2^jB)^{\frac{1}{2}-\frac{1}{p}},
\end{equation}
where $U_0(B)=\mathcal O(B)$, and $U_j(B)=\mathcal O(2^{j}B\setminus 2^{j-1}B)$ for $j=1,2,\dots$}.\label{numitem:2}
     \end{enumerate}
\end{definition}

\begin{lemma}\label{lem:L2molecule}
     There is a constant $C>0$ such that for any $j_0 \in \mathbb{Z}$, $j_0 \geq 0$ and for any $(p,2,M,\Delta_k,\varepsilon)$-molecule $\boldsymbol{m}$ associated with a ball $B$,  we have
    \begin{equation}\label{eq:m_on_U}
        \| \boldsymbol{m}\|^2_{L^2\big(\bigcup_{j\geq j_0} U_j(B)\big)} \leq C2^{-2\varepsilon j_0} w(2^{j_0}B)^{1-2/p}.
    \end{equation}
    In particular, using~\eqref{eq:m_on_U} for $j_0=0$, we get
    \begin{equation}\label{eq:L2Molecule}
        \| \boldsymbol{m}\|^2_{L^2(dw)} \leq Cw(B)^{1-2/p}.
    \end{equation}
\end{lemma}

\begin{proof}
    Applying condition~\eqref{eq:molecule_condition} and then~\eqref{eq:growth}, together with the fact that $1-2/p<0$, we get 
\begin{equation*}
\begin{split}
    \| m\|^2_{L^2(\bigcup_{j\geq j_0} U_j(B))}& =\sum_{j\geq j_0} \|\boldsymbol m\|_{L^2(U_j(B))}^2 \leq \sum_{j\geq j_0} 2^{-2j\varepsilon} w(2^jB)^{1-2/p}  \\&\leq C w(2^{j_0}B)^{1-2/p}\sum_{j \geq j_0}2^{-2\varepsilon j}2^{(j-j_0){N}(1-2/p)}\leq C 2^{-2\varepsilon j_0}w(2^{j_0}B)^{1-2/p}.
    \end{split} 
\end{equation*}
\end{proof}
Observe that any $(p,2,M,\Delta_k)$-atom is trivially a $(p,2,M,\Delta_k,\varepsilon)$-molecule for any $\varepsilon>0$. 

\begin{proposition}\label{prop:m_in_Hp}
    There is a constant $C=C(p,M,\varepsilon)$ such that 
    \[ \| \boldsymbol{m}\|_{\mathbb{H}^p_{\mathrm{Dunkl}}}\leq C\]
    for every $(p,2,M,\Delta_k,\varepsilon)$-molecule $\boldsymbol{m}$. 
\end{proposition}

\begin{proof}
     The proof mimics that of~\cite[Theorem 5.2]{Hofman}. Fix a $(p,2,M,\Delta_k,\varepsilon)$-molecule $\boldsymbol{m}$ associated with a ball $B=B(\mathbf{x}_0,r)$, and let $\boldsymbol{b}\in \mathcal{D}(\Delta_k^M)$ be such that $\boldsymbol{m}=\Delta_k^M \boldsymbol{b}$. We aim to estimate the $L^p(dw)$-norm of $S\boldsymbol{m}$ by splitting the domain of integration into a local region and its complement:
\begin{equation*}
    \|S\boldsymbol{m}\|_{L^p(dw)}^p = \int_{\mathcal{O}(B(\mathbf{x}_0,16r))} |S\boldsymbol{m}(\mathbf{x})|^p\,dw(\mathbf{x}) + \int_{\mathcal{O}(B(\mathbf{x}_0,16r))^c} |S\boldsymbol{m}(\mathbf{x})|^p\,dw(\mathbf{x}).
\end{equation*}

For the integral over  $\mathcal{O}(B(\mathbf{x}_0,16r))$, applying H\"older's inequality together with the $L^2(dw)$-boundedness of the operator $S$, the relation $w(\mathcal{O}(B(\mathbf{x}_0,16r))) \sim w(B(\mathbf{x}_0,r))$  (see~\eqref{eq:growth}), and  \eqref{eq:L2Molecule}, we obtain
\begin{equation*}
\begin{split}
    \int\limits_{\mathcal{O}(B(\mathbf{x}_0,16r))} |S\boldsymbol{m}(\mathbf{x})|^p\, dw(\mathbf{x}) &\leq \|S\boldsymbol{m}\|_{L^2(dw)}^p w\big(\mathcal{O}(B(\mathbf{x}_0,16r))\big)^{\frac{2-p}{2}}\leq C\|\boldsymbol{m}\|_{L^2(dw)}^p w(B)^{\frac{2-p}{2}} \leq C.
    \end{split}
\end{equation*}
 The proof will be complete if we show that there exists an $\alpha=\alpha(p,M,\varepsilon)>0$ such that 
    \begin{equation}\label{eq:Sm_part}
        \| S\boldsymbol{m}\|_{L^2(U_j(B))}\leq C2^{-\alpha j} w(2^jB)^{1/2-1/p} \text{ for all } j\geq 5. 
    \end{equation}
Then, by the Cauchy--Schwarz inequality, the fact that $w(U_j(B)) \sim w(2^jB)$, and~\eqref{eq:Sm_part}, we get
\begin{equation*}
    \begin{split}
        \int\limits_{\mathcal{O}(B(\mathbf{x}_0,16r))^c} |S\boldsymbol{m}(\mathbf{x})|^p\, dw(\mathbf{x}) & \leq \sum_{j=5}^{\infty} \int\limits_{U_j(B)} |S\boldsymbol{m}(\mathbf{x})|^p\, dw(\mathbf{x}) \leq \sum_{j=5}^{\infty} \|S\boldsymbol{m}\|_{L^2(U_j(B))}^p w\big(U_j(B)\big)^{\frac{2-p}{2}}\\&\leq C\sum_{j=5}^{\infty}2^{-\alpha j} \leq C.
    \end{split}
\end{equation*}
 In the proof we shall frequently use the relation $w(B(\mathbf{x},t))\sim w(B(\mathbf{y},t))$ for all $t>0$ and $\mathbf{x},\mathbf{y}\in \mathbb{R}^N$ such that $d(\mathbf{x},\mathbf{y})<t$ (see~\eqref{eq:growth}). Let us take $\theta \in (0,1)$, which will be chosen later. We write 
    \begin{equation}
        \begin{split}
           \|S\boldsymbol{m}\|_{L^2(U_j(B))}^2 & =\int_{U_j(B)}\int_{0}^{2^{\theta (j-5)}r} \int_{d(\mathbf{x},\mathbf{y})<t} |Q_t*\boldsymbol{m}(\mathbf{y})|^2\frac{dw(\mathbf{y})}{w(B(\mathbf{x},t))}\frac{dt}{t}dw(\mathbf{x})\\
            &\ + \int_{U_j(B)}\int_{2^{\theta (j-5)}r} ^\infty \int_{d(\mathbf{x},\mathbf{y})<t} |Q_t*\Delta_k^M \boldsymbol{b}(\mathbf{y})|^2\frac{dw(\mathbf{y})}{w(B(\mathbf{x},t))}\frac{dt}{t}dw(\mathbf{x})=:I+J.        \end{split}
    \end{equation}
    We start with $J$. First, observe that by Fubini's theorem,
    \begin{equation*}
        J \leq C \int_{2^{\theta (j-5)}r} ^\infty \int_{\mathbb{R}^N} |t^{2M}\Delta_k^{M}Q_t*\boldsymbol{b}(\mathbf{y})|^2dw(\mathbf{y})\frac{dt}{t^{4M+1}}.
    \end{equation*}
    Applying Lemma~\ref{lem:Q_operator_bounded}, condition~\eqref{eq:molecule_condition}, and~\eqref{eq:growth},
    \begin{equation*}
    \begin{split}
         J &\leq C \int_{2^{\theta (j-5)}r} ^\infty \|\boldsymbol{b}\|_{L^2(dw)}^2\frac{dt}{t^{4M+1}}\leq C 2^{-4M\theta(j-5)}r^{4M}r^{-4M} w(B)^{1-2/p}\\
            &= C 2^{-4M\theta(j-5)} w(2^jB)^{2/p-1}w(B)^{1-2/p} w(2^jB)^{1-2/p}\\
            &\leq C 2^{-4M\theta(j-5)}2^{j\mathbf{N}(2/p-1)}w(2^jB)^{1-2/p}=C 2^{-2\alpha j}w(2^jB)^{1-2/p},
         \end{split}
    \end{equation*}
    where $\alpha=2M\theta -\mathbf{N}\left(\frac{1}{p}-\frac{1}{2}\right)>0$, {provided $0<\theta<1$ is chosen sufficiently close to 1}.

    We turn to $I$. For $j \geq 5$, consider the sets 
    \begin{equation*}
        V_j=\{\mathbf{z}\in\mathbb{R}^N:d(\mathbf{x}_0,\mathbf{z})\geq 2^{j-3}r\}, \quad W_j=\{ \mathbf{z}\in\mathbb{R}^N: d(\mathbf{x}_0,\mathbf{z})< 2^{j-3}r\}.
    \end{equation*}
    We split
    \begin{equation*}
        \boldsymbol{m}=\boldsymbol{m}\chi_{V_j}+\boldsymbol{m}\chi_{W_j}:=\boldsymbol{m}_1+\boldsymbol{m}_2.
    \end{equation*}
We start with the term involving $\boldsymbol{m}_1$.     By Fubini's theorem, we get
     \begin{equation*}
\begin{split}
    I_1:&=\int_{U_j(B)}\int_{0}^{2^{\theta (j-5)}r} \int_{d(\mathbf{x},\mathbf{y})<t} |Q_t*\boldsymbol{m}_1(\mathbf{y})|^2\frac{dw(\mathbf{y})}{w(B(\mathbf{x},t))}\frac{dt}{t}dw(\mathbf{x}) \\
    &\leq C \int_{0}^{2^{\theta (j-5)}r} \int_{\mathbb{R}^N} |Q_t*\boldsymbol{m}_1(\mathbf{y})|^2 \left( \int_{d(\mathbf{x},\mathbf{y})<t} \frac{dw(\mathbf{x})}{w(B(\mathbf{y},t))} \right) dw(\mathbf{y}) \frac{dt}{t} \\
    &\leq C \int_{0}^{2^{\theta (j-5)}r} \int_{\mathbb{R}^N} |Q_t*\boldsymbol{m}_1(\mathbf{y})|^2\,dw(\mathbf{y})\,\frac{dt}{t}.
\end{split} 
\end{equation*}
     Then, by Lemma~\ref{lem:eta} applied to $\eta_t=Q_t$ and Lemma~\ref{lem:L2molecule} we get
     \begin{equation*}
      I_1\leq    \int_{0}^{2^{\theta (j-5)}r} \int_{\mathbb{R}^N} |Q_t*\boldsymbol{m}_1(\mathbf{y})|^2\,dw(\mathbf{y})\,\frac{dt}{t} \leq C\|\boldsymbol{m}_1\|_{L^2(dw)}^2 \leq C2^{-2\varepsilon j}w(2^jB)^{1-2/p}.
     \end{equation*}
      To estimate the term involving $\boldsymbol{m}_2$, we first note that for $\mathbf{x}\in U_j(B)$, $d(\mathbf{x},\mathbf{y})<t\leq 2^{\theta(j-5)}r$ and     
      $\mathbf{z}\in W_j$, we have $ d(\mathbf y,\mathbf z)>2^{j-2}r$. Hence, for such points, by \eqref{eq:Qt-bound},  we obtain 
     \begin{equation*}
             |Q_t(\mathbf{y},\mathbf{z})|\leq C w(B(\mathbf{x},t))^{-1} \exp(-c2^jr/t).
     \end{equation*}
      Consequently, for $\mathbf x\in U_j(B)$ and $d(\mathbf{x},\mathbf{y})\leq t\leq 2^{\theta(j-5)}r$, we have
      \begin{equation*}
      \begin{split}
          |Q_t*\boldsymbol{m}_2(\mathbf{y})| & \leq C w(B(\mathbf{x},t))^{-1}\exp(-c2^jr/t)\|\boldsymbol{m}_2\|_{L^1(dw)}  \\
          &\leq C w(B(\mathbf{x},t))^{-1}\exp(-c2^jr/t)w(2^{j-3}B)^{1/2}\|\boldsymbol{m}_2\|_{L^2(dw)}\\
          &\leq C w(B(\mathbf{x},t))^{-1}\exp(-c2^jr/t)w(2^{j-3}B)^{1/2}w(B)^{1/2-1/p},
      \end{split} \end{equation*}
      where we have used the Cauchy--Schwarz inequality in the second step, and \eqref{eq:m_on_U} in the last one. 
      Fixing an integer $K \in \mathbb{N}$ to be chosen later, we obtain 
 \begin{equation*}
 \begin{split}
     |Q_t*\boldsymbol m_2(\mathbf y)|&\leq C_K  w(B(\mathbf{x},t))^{-1}\Big(\frac{t}{2^jr}\Big)^K w(2^{j-3}B)^{1/2}w(B)^{1/2-1/p}\\
     &\leq C_K \frac{w(B(\mathbf x, 2^jr))}{w(B(\mathbf x,t))}\frac{1}{w(B(\mathbf x,2^jr))} \Big(\frac{t}{2^jr}\Big)^Kw(2^jB)^{1/2} \frac{w(2^jB)^{1/p-1/2}}{{w(B)}^{1/p-1/2}} w(2^jB)^{1/2-1/p}.
     \end{split} 
 \end{equation*}
      Observe that  $w(B(\mathbf x,2^jr))\sim w(2^jB)\sim w(U_j(B))$ for $\mathbf x\in U_j(B)$. Thus,  applying \eqref{eq:growth}, we get 
      \begin{equation*}
          |Q_t*\boldsymbol m_2(\mathbf y)|\leq C_K w(U_j(B))^{-1/2}\left(\frac{2^jr}{t}\right)^{\mathbf N-K} 2^{j(1/p-1/2)\mathbf N} w(2^jB)^{1/2-1/p} \quad \text{ for } d(\mathbf x,\mathbf y)<t\leq 2^{\theta (j-5)}r. 
      \end{equation*}
    Finally, 
    
      \begin{equation*}
      \begin{split}
          I_2&=\int_{U_j(B)}\int_{0}^{2^{\theta (j-5)}r} \int_{d(\mathbf{x},\mathbf{y})<t} |Q_t*\boldsymbol{m}_2(\mathbf{y})|^2 \frac{dw(\mathbf{y})}{w(B(\mathbf{x},t))} \frac{dt}{t} dw(\mathbf{x}) \\
          &\leq C_K w(U_j(B))^{-1} w(2^jB)^{1-\frac{2}{p}} 2^{2j(1/p-1/2)\mathbf N} \int_{U_j(B)} \left( \int_{0}^{2^{\theta (j-5)}r} \left(\frac{2^jr}{t}\right)^{2\mathbf{N}-2K} \frac{dt}{t} \right) dw(\mathbf{x}) \\
          &\leq C_K w(2^jB)^{1-\frac{2}{p}} 2^{j\mathbf{N}(\frac{2}{p}-1)} 2^{-j(2K-2\mathbf{N})(1-\theta)}.
      \end{split}
      \end{equation*}
 Now, choosing $K$ large enough, we obtain the desired bound $I_2\leq C2^{-2\alpha j}w(2^jB)^{1-2/p}$. 
\end{proof}

We conclude this section with a proposition that will be used later.

For $\mathbf{n} \in \mathbb{N}$, we define a semi-norm on $\mathcal S(\mathbb{R}^N)$:
\begin{equation*}
        \|f\|_{(\mathbf{n})}=\sup_{\mathbf{x}\in\mathbb{R}^N, |\boldsymbol \alpha|\leq \mathbf{n}}\{ (1+\|\mathbf{x}\|)^{\mathbf{n}} |\partial^{\boldsymbol\alpha}f(\mathbf{x})|\}.
    \end{equation*} 
 
\begin{proposition}\label{lem:atom_S}
    Fix $0<p\leq 1$ and $M>\mathbf{N}(2-p)/4p$. Let $\mathbf{n}>\max (2M,\mathbf{N})$ be a positive integer. There is  a constant $C>0$ such that for all $f \in \mathcal{S}(\mathbb{R}^N)$ and every $ (p,2,M,\Delta_k)$-atom $\boldsymbol{a}(\cdot)$, we have 
    \begin{equation*}
        \Big|\int_{\mathbb{R}^N} \boldsymbol{a}(\mathbf{x}) f(\mathbf{x})\, dw(\mathbf{x})\Big|\leq C\|f\|_{(\mathbf{n})}.
    \end{equation*}
\end{proposition}

We need the following lemma, whose proof we include in the appendix. 

\begin{lemma}\label{lem:L_2_Schwartz}
    Let $M_1\geq 0$, $M_1 \in \mathbb{Z}$, and  $\mathbf{n}>\max\left(\frac{\mathbf{N}}{2},2M_1\right)$. Then, there is a constant $C>0$ such that for all $f \in \mathcal{S}(\mathbb{R}^N)$ we have
    \begin{equation*}
        \|(\Delta_k)^{M_1}f\|_{L^2(dw)} \leq C\|f\|_{(\mathbf{n})}.
    \end{equation*}
\end{lemma}

\begin{proof}[Proof of Proposition \ref{lem:atom_S}]

It follows from \eqref{eq:balls_asymp} that there is a constant $c>0$ such that for all $r>0$ and $\mathbf{x}_0 \in \mathbb{R}^N$ we have
\begin{equation*}
    cr^{\mathbf{N}}=w(B(0,r))\leq w(B(\mathbf{x}_0,r)). 
\end{equation*}
     Let $\boldsymbol{a}$ be a $(p,2,M,\Delta_k)$-atom   associated with a ball $B(\mathbf{x}_0,r)$. We consider two cases.

     \medskip
     \noindent\textbf{Case 1:} $r \geq 1$. Applying the Cauchy--Schwarz inequality and Lemma~\ref{lem:L_2_Schwartz} with $M_1=0$,  we get
    \begin{equation*}
       \Big|\int_{\mathbb{R}^N} \boldsymbol{a}(\mathbf{x}) f(\mathbf{x})\, dw(\mathbf{x})\Big|\leq \|\boldsymbol{a}\|_{L^2(dw)}  \|f\|_{L^2(dw)}\leq Cw(B(\mathbf{x}_0,r))^{\frac{1}{2}-\frac{1}{p}} \|f\|_{(\mathbf{n})}\leq C\|f\|_{(\mathbf{n})}, 
    \end{equation*}
   since $w(B(\mathbf{x}_0,r))\geq c$ and  $\frac{1}{2}-\frac{1}{p}<0$. 

 \medskip
     \noindent\textbf{Case 2:} $0<r < 1$.   
     By the definition of a $(p,2,M,\Delta_k)$-atom, there is $\boldsymbol{b}\in \mathcal{D}(\Delta_k^M)$ satisfying~\eqref{eq:support} and~\eqref{eq:size} such that $\boldsymbol{a}=\Delta_k^M{\boldsymbol{b}}$. Hence, applying the Cauchy--Schwarz inequality, self-adjointness of $\Delta_k$, and Lemma~\ref{lem:L_2_Schwartz} with $M_1=M$, we obtain
     \begin{equation*}
     \begin{split}
       \Big|\int_{\mathbb{R}^N} \boldsymbol{a}(\mathbf{x}) f(\mathbf{x})\, dw(\mathbf{x})\Big|&=\Big| \int_{\mathbb{R}^N} \Delta_k^M\boldsymbol{b}(\mathbf{x})f(\mathbf{x})\, dw(\mathbf{x})\Big|\\
     &=\Big| \int_{\mathbb{R}^N} \boldsymbol{b}(\mathbf{x})\Delta_k^Mf(\mathbf{x})\, dw(\mathbf{x})\Big|\leq 
     \|\boldsymbol{b}\|_{L^2(dw)} \|\Delta_k^M f\|_{L^2(dw)}\\  
    &  \leq Cr^{2M}w(B(\mathbf{x}_0,r))^{\frac{1}{2}-\frac{1}{p}} \|f\|_{(\mathbf{n})} \leq Cr^{2M}r^{\mathbf{N}(\frac{1}{2}-\frac{1}{p})}\|f\|_{(\mathbf{n})}\leq C \|f\|_{(\mathbf{n})}, 
    \end{split} \end{equation*}
    where in the last inequality we have used   $0<r < 1$ and  the assumption $M>\mathbf{N}(2-p)/4p$, which guarantees that the total exponent of $r$ is strictly positive.  
\end{proof}

\subsection{Definition of the atomic Hardy space}\label{section-atomic}

We begin this part with the following proposition, which will be used in the definition of the $H^p_{M-{\rm atom}}$ space. 

\begin{proposition}
\label{prop:dist_convergence} Let $0<p\leq 1$ and $M \in \mathbb{Z}$ be such that $M>\mathbf{N}(2-p)/4p$. 
    For each $n \in \mathbb{N}$, let $\lambda_n\in\mathbb C$ and let $\boldsymbol{a}_n$ be a $(p,2,M,\Delta_k)$-atom. If $\sum_{n=1}^{\infty}|\lambda_n|^p<\infty$, then the series $\sum_{n=1}^\infty \lambda_n \boldsymbol{a}_n$
    converges (unconditionally) in $\mathcal S'(\mathbb{R}^N)$.
\end{proposition}

\begin{proof} 
    Let $\mathbf n$ and $C>0$ be as in  Lemma \ref{lem:atom_S}. Then,   for any $\varphi \in \mathcal{S}(\mathbb{R}^N)$, we have
    \begin{equation*}
        \sum_{n=1}^{\infty}|\lambda_n|\left|\langle \boldsymbol{a}_n,\varphi\rangle\right| \leq C\Big(\sum_{n=1}^{\infty}|\lambda_n|\Big)\|\varphi\|_{(\mathbf n)}\leq C\Big(\sum_{n=1}^{\infty}|\lambda_n|^p\Big)^{1/p}\|\varphi\|_{(\mathbf n)},
    \end{equation*}
    which implies the proposition. 
\end{proof}

\begin{definition}\label{def:atomic_Hardy}
    Suppose $0 < p \leq 1$, $M > \mathbf{N}(2-p)/4p$, and $M \in \mathbb{Z}$. The atomic Hardy space $H^p_{M\text{-}\mathrm{atom}}$ is defined as the set of all distributions $f \in \mathcal{S}'(\mathbb{R}^N)$ that can be represented as
    \[
        f = \sum_{j=1}^{\infty} \lambda_j \boldsymbol{a}_j,
    \]
    where the series converges in $\mathcal{S}'(\mathbb{R}^N)$, $\lambda_j \in \mathbb{C}$ with $\sum_{j=1}^{\infty} |\lambda_j|^p < \infty$, and each $\boldsymbol{a}_j$ is a $(p,2,M,\Delta_k)\text{-atom}$.
    
    For such $f$, we define 
    \[
        \| f\|_{H^p_{M\text{-}\mathrm{atom}}}^p = \inf \sum_{j=1}^{\infty} |\lambda_j|^p,
    \]
    where the infimum is taken over all representations of $f$ of the form above.
\end{definition}

The principal objective of Section~\ref{sec:Hardy} is to establish the following result.

\begin{theorem}\label{teo:H_p_coincides}
     Suppose $0 < p \leq 1$, $M > \mathbf{N}(2-p)/4p$, and $M \in \mathbb{Z}$. Then the spaces $H^p_{\mathrm{Dunkl}}$ and $H^p_{M\mathrm{-atom}}$ coincide, and their respective quasi-norms, $\|\cdot \|_{H^p_{\mathrm{Dunkl}}}$ and $\|\cdot \|_{H^p_{M\mathrm{-atom}}}$, are equivalent. 
\end{theorem}

\subsection{Atomic decompositions of  functions from $\mathbb{H}^p_{\mathrm{Dunkl}}$}

Our next step is to establish atomic decompositions for functions in $\mathbb{H}^p_{\mathrm{Dunkl}}$. We will then extend this result to all remaining elements of $H^p_{\mathrm{Dunkl}}$ in Subsection~\ref{sec:atomic_proof}.

\begin{proposition}\label{propo:atomic_decomposition}
Fix $0<p\leq 1$, $M>\mathbf N(2-p)/4p$, $M\in\mathbb Z$. Then there exists a constant $C>0$ such that 
for all  $f\in \mathbb{H}^p_{\mathrm{Dunkl}}$  there exist a sequence $(\lambda_j)_{j \in \mathbb{N}}$ of complex numbers and a sequence $(\boldsymbol{a}_j)_{j \in \mathbb{N}}$ of $(p,2,M,\Delta_k)$-atoms such that 
\begin{equation} 
f(\mathbf{x})=\sum_{j=1}^{\infty} \lambda_j \boldsymbol{a}_j(\mathbf{x}),\quad \sum_{j=1}^{\infty} |\lambda_j|^p\leq C \|Sf\|_{L^p(dw)}^p=C\|f\|_{\mathbb{H}^p_{\mathrm{Dunkl}}}^p,
\end{equation} 
where the series $\sum_{j=1}^{\infty} \lambda_j \boldsymbol{a}_j$ converges in $L^2(dw)$ and $\mathcal{S}'(\mathbb{R}^N)$. 
\end{proposition}

\begin{proof}
    The proof follows the approach in~\cite{Hofman} (see also~\cite{Duong},~\cite{ADzH}). Set $F(t,\mathbf{x})=t^2(\Delta_k h_{t^2})*f(\mathbf{x})$. Since $f \in \mathbb{H}^p_{\mathrm{Dunkl}}$ (so, in particular $f \in L^2(dw)$), we have $F \in T_2^2 \cap T^p_2$. Applying the atomic decomposition in $T_2^p$ (see Theorem~\ref{teo:atomic_decomposition_tent}), we see that there are non-negative real numbers $\lambda_j$ and $T_2^p$-atoms $A_j(t,\mathbf{x})$ such that 
    \begin{equation}\label{eq:tent_series}
        F(t,\mathbf{x})=\sum_{j=1}^{\infty} \lambda_j A_j(t,\mathbf{x}), \quad 
        \sum_{j=1}^{\infty}|\lambda_j|^p\leq C\|F\|_{T_2^p}^p=C\|Sf\|_{L^p(dw)}^p.
    \end{equation}
    Moreover, due to Remark~\ref{remarkAtom} and the fact that $F \in T_2^2 \cap T^p_2$, the series in~\eqref{eq:tent_series} can be chosen in such a way that the convergence is in $T^p_2$, $T_2^2$ (consequently, in $L^2(dw \,dt/t)$), and pointwise. Let $\phi,\psi$ be as in~\eqref{eq:phi_form}. Let $A_j$ be an atom of the tent space $T^p_2$, which appears in~\eqref{eq:tent_series}. Assume that it is associated with a tent $\widehat{B}$, $B=B(\mathbf{x}_0,r)$. We define
    \begin{align*}
      \boldsymbol{a}_j(\mathbf{x})&:=\Pi_{\phi*\phi} (A_j)(\mathbf{x})=\int_0^\infty\int_{\mathbb{R}^N} \phi_t^{*2} (\mathbf{x},\mathbf{y})A_j(t,\mathbf{y})\, dw(\mathbf{y})\frac{dt}{t}\\
      &= \int_0^\infty\int_{\mathbb{R}^N} t^{2M} \Delta_k^{M} (\psi_t*\phi_t) (\mathbf{x},\mathbf{y})A_j(t,\mathbf{y})\, dw(\mathbf{y})\frac{dt}{t},\\[1ex]
      \boldsymbol{b}_j(\mathbf{x})&:=\int_0^\infty\int_{\mathbb{R}^N} t^{2M} (\psi_t*\phi_t) (\mathbf{x},\mathbf{y})A_j(t,\mathbf{y})\, dw(\mathbf{y})\frac{dt}{t}.
\end{align*}
     To check whether $\boldsymbol{a}_j$ is a constant multiple of a $(p,2,M,\Delta_k)$-atom associated with $B(\mathbf{x}_0,8r)$, we will analyze the function $\boldsymbol{b}_j$.

      Let us verify condition~\eqref{eq:cancellation}. Since $A_j \in L^2(dw \, dt/t)$ and it is compactly supported, we have
      \begin{align*}
          \Delta_k^{M} \boldsymbol{b}_j(\mathbf{x})=\Delta_k^{M} \int_0^\infty\int_{\mathbb{R}^N} (\psi_t*\phi_t) (\mathbf{x},\mathbf{y})t^{2M} A_j(t,\mathbf{y})\, dw(\mathbf{y})\frac{dt}{t}=\boldsymbol{a}_j(\mathbf{x}).
      \end{align*}

      Let us verify condition~\eqref{eq:support}. Since $\supp A_j \subseteq \widehat B$, by the definition of $\widehat B$ (see~\eqref{eq:tent_def}), we have
      \begin{equation}\label{eq:supp_restrict}
          \int_0^\infty\int_{\mathbb{R}^N} t^{2M} (\psi_t*\phi_t) (\mathbf{x},\mathbf{y})A_j(t,\mathbf{y})\, dw(\mathbf{y})\frac{dt}{t}=\int_0^{r}\int_{B(\mathbf{x}_0,r)} t^{2M}  (\psi_t*\phi_t) (\mathbf{x},\mathbf{y})A_j(t,\mathbf{y})\, dw(\mathbf{y})\frac{dt}{t}
      \end{equation}
    Since $\supp \psi,\supp \phi \subseteq B(0,1)$, we have $\supp (\psi_t*\phi_t) \subseteq B(0,2t) \subseteq B(0,2r)$. Then, it follows from the formula~\eqref{eq:translation-radial} (see also Theorem~\ref{teo:support}) that $\supp (\psi_t*\phi_t)(\cdot,\mathbf{y}) \subseteq \mathcal{O}(B(\mathbf{y},2r))$. Consequently, by the formula~\eqref{eq:supp_restrict} we see that $\boldsymbol{b}_j(\mathbf{x})=0$ for $\mathbf{x} \not\in \mathcal{O}(B(\mathbf{x}_0, 8r))$. Therefore, $\supp \boldsymbol{b}_j \subseteq \mathcal{O}(B(\mathbf{x}_0,8r))$.

    Finally, let us verify the size condition~\eqref{eq:size}. Let $n \in \{0,1,\ldots,M\}$. We have
    \begin{align*}
        &\|(r^2\Delta_k)^n\boldsymbol{b}_j\|_{L^2(dw)}\\&=\sup_{g \in L^2(dw), \, \|g\|_{L^2(dw)}=1}\left|\int_{\mathbb{R}^N}g(\mathbf{x})\int_0^\infty\int_{\mathbb{R}^N} (r^2\Delta_k)^nt^{2M}(\psi_t*\phi_t)(\mathbf{x},\mathbf{y})A_j(t,\mathbf{y})\,dw(\mathbf{y})\,\frac{dt}{t}\,dw(\mathbf{x})\right|.
    \end{align*}
    Let us denote $\eta(\mathbf{x})=(\Delta_k^n \psi)*\phi$. Clearly, $\eta \in \mathcal{S}(\mathbb{R}^N)$. By the scaling property of the Dunkl Laplacian we have $\Delta_k^n(\psi_t*\phi_t)=t^{-2n}\eta_t$, hence,  using the symmetry of the generalized convolution (see~\eqref{eq:conv_translation}) we obtain
    \begin{align*}
        \|(r^2\Delta_k)^n\boldsymbol{b}_j\|_{L^2(dw)}=\sup_{g \in L^2(dw), \, \|g\|_{L^2(dw)}=1}r^{2n}\left|\int_0^{\infty}\int_{\mathbb{R}^N}\eta_t*g(\mathbf{y})t^{2M-2n}A_j(t,\mathbf{y})\,dw(\mathbf{y})\frac{dt}{t}\right|.
    \end{align*}
    Applying the Cauchy--Schwarz inequality and Lemma~\ref{lem:eta}, we get
    \begin{align*}
        &\|(r^2\Delta_k)^n\boldsymbol{b}_j\|_{L^2(dw)} \leq r^{2n}\left(\int_0^{\infty}\int_{\mathbb{R}^N}|\eta_t*g(\mathbf{y})|^2\,dw(\mathbf{y})\frac{dt}{t}\right)^{\frac{1}{2}} \left(\int_0^{\infty}\int_{\mathbb{R}^N}t^{2(2M-2n)}|A_j(t,\mathbf{y})|^2\,dw(\mathbf{y})\frac{dt}{t}\right)^{\frac{1}{2}}\\&\leq Cr^{2n}\left(\int_0^{\infty}\int_{\mathbb{R}^N}t^{2(2M-2n)}|A_j(t,\mathbf{y})|^2\,dw(\mathbf{y})\frac{dt}{t}\right)^{\frac{1}{2}}.
    \end{align*}
    Finally, by the fact that $\supp A_j \subseteq \widehat B$ and the definition of the atom of $T_2^p$ (see Definition~\ref{def:tent_atom_def}),
    \begin{align*}
        &r^{2n}\left(\int_0^{\infty}\int_{\mathbb{R}^N}t^{2(2M-2n)}|A_j(t,\mathbf{y})|^2\,dw(\mathbf{y})\frac{dt}{t}\right)^{\frac{1}{2}}=r^{2n}\left(\int_0^{r}\int_{\mathbb{R}^N}t^{2(2M-2n)}|A_j(t,\mathbf{y})|^2\,dw(\mathbf{y})\frac{dt}{t}\right)^{\frac{1}{2}}\\&\leq Cr^{2M}\left(\int_0^{r}\int_{\mathbb{R}^N}|A_j(t,\mathbf{y})|^2\,dw(\mathbf{y})\frac{dt}{t}\right)^{\frac{1}{2}} \leq Cr^{2M}w(B)^{\frac{1}{2}-\frac{1}{p}},
    \end{align*}
    so the size condition is verified. Since $\boldsymbol{a}_j = \Delta_k^{M} \boldsymbol{b}_j$, and $\boldsymbol{b}_j$ satisfies the required support and size conditions, we conclude that $\boldsymbol{a}_j$ is a constant multiple of a $(p,2,M,\Delta_k)$-atom. Finally, let us recall that the convergence of the series~\eqref{eq:tent_series} is (in particular) in $T_2^2$. Therefore, due to~\eqref{L2backL2}, we have 
\begin{equation*}
        f(\mathbf{x}) =\Pi_{\phi * \phi} (F)(\mathbf{x})= \sum_{j=1}^{\infty} \lambda_j \Pi_{\phi*\phi} (A_j)(\mathbf{x})=\sum_{j=1}^{\infty} \lambda_j \boldsymbol{a}_j(\mathbf{x}),
\end{equation*}
where the convergence is in $L^2(dw)$. Hence, the convergence is in the sense of tempered distributions as well. The decomposition for $f\in\mathbb{H}^p_{\mathrm{Dunkl}}$ is obtained with 
\begin{equation*}
    \sum_{j=1}^{\infty} |\lambda_j|^p\leq C \|S f\|^p_{L^p(dw)}.
\end{equation*}
Since each $\boldsymbol{a}_j$ is a constant multiple of a $(p,2,M,\Delta_k)$-atom, i.e., $\boldsymbol{a}_j = C \widetilde{\boldsymbol{a}}_j$ for a $(p,2,M,\Delta_k)$-atom $\widetilde{\boldsymbol{a}}_j$, we set $\widetilde{\lambda}_j = C \lambda_j$. The exact decomposition takes the form $f = \sum_{j=1}^{\infty} \widetilde{\lambda}_j \widetilde{\boldsymbol{a}}_j$, and the new coefficients satisfy the required bound $ \sum_{j=1}^{\infty} |\widetilde{\lambda}_j|^p  \leq C \|S f\|^p_{L^p(dw)} $, which  completes the proof.
\end{proof}

\subsection{Identification of the elements of $H^p_{\mathrm{Dunkl}}$ with tempered distributions}

\begin{corollary}\label{distr-bound}
For $0<p\leq 1$, let $M>\mathbf{N}(2-p)/4p$ and  $\mathbf{n}>\max (2M,\mathbf{N})$. Then     there is a constant $C>0$ such that for all $f\in\mathbb{H}^p_{\mathrm{Dunkl}}$ we have 
\begin{equation}\label{eq:quantitive}
    \Big| \int_{\mathbb{R}^N} f(\mathbf{x})\varphi (\mathbf{x})\, dw(\mathbf{x})\Big|\leq C \| Sf\|_{L^p(dw)} \|\varphi\|_{(\mathbf{n})}=C \|f\|_{\mathbb{H}^p_{\mathrm{Dunkl}}}\|\varphi \|_{(\mathbf{n})}.
\end{equation}
\end{corollary}

\begin{proof}
    From Proposition~\ref{propo:atomic_decomposition}, $f$ admits a representation in the form $f=\sum_{j=1}^{\infty}\lambda_j \boldsymbol{a}_j$, where $\lambda_j \in \mathbb{C}$ and $\boldsymbol{a}_j$ are $(p,2,M,\Delta_k)$-atoms, the series converges in $L^2(dw)$ and $\mathcal{S}'(\mathbb{R}^N)$, and $\sum_{j=1}^{\infty}|\lambda_j|^p \leq C\|f\|_{\mathbb{H}^p_{\mathrm{Dunkl}}}^p$. Consequently,
    \begin{align*}
         \Big| \int_{\mathbb{R}^N} f(\mathbf{x})\varphi (\mathbf{x})\, dw(\mathbf{x}) \Big|\leq \sum_{j=1}^{\infty}|\lambda_j|\Big| \int_{\mathbb{R}^N} \boldsymbol{a}_j(\mathbf{x})\varphi (\mathbf{x})\, dw(\mathbf{x}) \Big|.
    \end{align*}
    By Lemma~\ref{lem:atom_S}, we get
    \begin{align*}
        \sum_{j=1}^{\infty}|\lambda_j|\Big| \int_{\mathbb{R}^N} \boldsymbol{a}_j(\mathbf{x})\varphi (\mathbf{x})\, dw(\mathbf{x}) \Big| \leq C\|\varphi\|_{(\mathbf{n})}\sum_{j=1}^{\infty}|\lambda_j|\leq C\|\varphi\|_{(\mathbf{n})}\Big(\sum_{j}|\lambda_j|^p\Big)^{1/p}.
    \end{align*} 
\end{proof}

\begin{corollary}\label{coro:distribution}
    Assume that $(f_n)_{n \in \mathbb{N}}$, $f_n\in\mathbb{H}^p_{\mathrm{Dunkl}}$, is a Cauchy sequence in the $H_{\mathrm{Dunkl}}^p$-norm.  Then $f_n$ converges in $\mathcal S'(\mathbb{R}^N)$ to a distribution $f\in\mathcal S'(\mathbb{R}^N)$. 
    
\end{corollary}

\begin{proof}
    The convergence in $\mathcal S'(\mathbb{R}^N)$ of the sequence $f_n$ to a distribution $f$ 
     is a consequence of the quantitative estimate~\eqref{eq:quantitive} from Corollary~\ref{distr-bound}.
\end{proof}

As a consequence of Corollary~\ref{coro:distribution}, the mapping from $H^p_{\mathrm{Dunkl}}$ (defined as equivalence classes of Cauchy sequences $(f_n)$ in  $ H^p_{\mathrm{Dunkl}}$ equipped with the norm $\|[(f_n)_{n \in \mathbb{N}}]\|_{H^p_{\rm Dunkl}}=\lim_{n \to \infty} \|f_n\|_{\mathbb H^p_{\rm Dunkl}}$) to $\mathcal{S}'(\mathbb{R}^N)$ is well-defined. To show that the elements of $H^p_{\mathrm{Dunkl}}$ can be uniquely identified with those in $\mathcal{S}'(\mathbb{R}^N)$, we must prove that this mapping is also injective. This is included in Proposition~\ref{propo:not_zero_limit}.

Now, let $(f_n)_{n \in \mathbb{N}}$ be a Cauchy sequence in $\mathbb{H}^p_{\mathrm{Dunkl}}$ and let  $f \in \mathcal{S}'(\mathbb{R}^N)$ be its distributional limit. We define
\begin{equation}\label{eq:pointwise_Qt}
    Q_t*f(\mathbf{x})=\lim_{n \to \infty}\int_{\mathbb{R}^N}f_n(\mathbf{y})Q_t(\mathbf{x},\mathbf{y})\,dw(\mathbf{y})=\lim_{n \to \infty}Q_t*f_n(\mathbf{x}).
\end{equation}
The quantitative estimate~\eqref{eq:quantitive} together with the properties of $Q_t$ (see Subsection~\ref{sec:Qt}) guarantees that $Q_t*f(\mathbf{x})$ is well-defined for all $(t,\mathbf{x}) \in (0,\infty) \times \mathbb{R}^N$. 

    \begin{proposition}\label{propo:not_zero_limit}
        Let $(f_n)_{n \in \mathbb{N}}$ be a Cauchy sequence in $\mathbb{H}^p_{\mathrm{Dunkl}}$. Let  $F(t,\mathbf{x})=Q_t*f(\mathbf x)$, where $Q_t*f(\mathbf{x})$ is defined in~\eqref{eq:pointwise_Qt}. Then $F\in T_2^p$, and $\|F\|_{T_2^p}=\lim_{n\to\infty} \| f_n\|_{\mathbb{H}^p_{\mathrm{Dunkl}}}$.
    \end{proposition}

\begin{proof}
     Now, let us show that the function $(t,\mathbf{x}) \to Q_t*f(\mathbf{x})$ belongs to $T_2^p$. For that purpose, we will imitate a standard proof of completeness for function spaces. First, without loss of generality (we can pass to a subsequence if necessary), we may assume
    \begin{equation}\label{eq:Cauchy_ver}
        \| f_n-f_m\|_{\mathbb{H}^p_{\mathrm{Dunkl}}}\leq C_1 2^{-n}\quad \text{ for } 1\leq n\leq m,
    \end{equation}
    and $f_0 \equiv 0$. Let $F_n(t,\mathbf{x})=Q_t*f_n(\mathbf{x})$. Let us define an auxiliary function
    \begin{equation*}
        G(t,\mathbf{x})=\sum_{n=1}^{\infty} |F_n(t,\mathbf{x})-F_{n-1}(t,\mathbf{x})|.
    \end{equation*}
    Let us check if $G \in T_2^p$. First, observe that by the Minkowski inequality, for all $\mathbf{x} \in \mathbb{R}^N$ we have
    \begin{align*}
        \mathcal{A}G(\mathbf{x})&=\left(\int_0^{\infty}\int_{\|\mathbf{x}-\mathbf{y}\|<t}\left(\sum_{n=1}^{\infty}|F_n(t,\mathbf{y})-F_{n-1}(t,\mathbf{y})|\right)^2\,\frac{dw(\mathbf{y})}{w(B(\mathbf{x},t))}\,\frac{dt}{t}\right)^{\frac{1}{2}}\\&\leq \sum_{n=1}^{\infty} \left(\int_0^{\infty}\int_{\|\mathbf{x}-\mathbf{y}\|<t}|F_n(t,\mathbf{y})-F_{n-1}(t,\mathbf{y})|^2\,\frac{dw(\mathbf{y})}{w(B(\mathbf{x},t))}\,\frac{dt}{t}\right)^{\frac{1}{2}}=\sum_{n=1}^{\infty}\mathcal{A}(F_n-F_{n-1})(\mathbf{x}).
    \end{align*}
    Therefore, using the numerical inequality $\left(\sum_{n=1}^{\infty}|\boldsymbol{a}_n|\right)^p \leq \sum_{n=1}^{\infty}|\boldsymbol{a}_n|^p$  and~\eqref{eq:Cauchy_ver} we obtain
    \begin{align*}
        \|\mathcal{A}G\|_{L^p}^{p} &\leq \int_{\mathbb{R}^N}\Big|\sum_{n=1}^{\infty}\mathcal{A}(F_{n}-F_{n-1})(\mathbf{x})\Big|^p\,dw(\mathbf{x}) \leq \sum_{n=1}^{\infty}\int_{\mathbb{R}^N}\left|\mathcal{A}(F_n-F_{n-1})(\mathbf{x})\right|^p\,dw(\mathbf{x})\\& \leq \sum_{n=1}^{\infty}\|f_n-f_{n-1}\|_{\mathbb{H}^p_{\mathrm{Dunkl}}}^p \leq C\sum_{n=1}^{\infty}2^{-(n-1)p} \leq C.
    \end{align*}
    Next, by the fact that $f_0 \equiv 0$ and the telescopic sum $F_n=\sum_{j=1}^{n}(F_j-F_{j-1})$, 
    \begin{equation}\label{eq:f_n_by_Qt}
        |F_n(t,\mathbf{x})| \leq G(t,\mathbf{x}) \text{ for all }t>0\text{ and }\mathbf{x} \in \mathbb{R}^N.
     \end{equation}
    Moreover, passing to the limit as $n \to \infty$ in~\eqref{eq:f_n_by_Qt} and using the pointwise convergence from~\eqref{eq:pointwise_Qt}, we immediately obtain
\begin{equation*}
    |F(t,\mathbf{x})|=|Q_t * f(\mathbf{x})| \leq G(t, \mathbf{x}) \quad \text{for all } t>0 \text{ and } \mathbf{x} \in \mathbb{R}^N.
\end{equation*}
     Consequently,
     \begin{equation}\label{eq:AF_n}
     \begin{split}
         \mathcal{A}(F_n-F)(\mathbf{x})^2 &=\int_0^{\infty}\int_{\|\mathbf{x}-\mathbf{y}\|<t}|(F_n-F)(t,\mathbf{y})|^2\frac{dw(\mathbf{y})}{w(B(\mathbf{x},t))}\,\frac{dt}{t}\\&\leq 4\int_0^{\infty}\int_{\|\mathbf{x}-\mathbf{y}\|<t}|G(t,\mathbf{y})|^2\,\frac{dw(\mathbf{y})}{w(B(\mathbf{x},t))}\,\frac{dt}{t}=\mathcal{A}(2G)(\mathbf{x})^2,
        \end{split}
     \end{equation}
     so, 
     \begin{align*}
         \int_{\mathbb{R}^N}\mathcal{A}(F_n-F)(\mathbf{x})^p \,dw(\mathbf{x}) \leq \int_{\mathbb{R}^N}\mathcal{A}(2G)(\mathbf{x})^p\,dw(\mathbf{x}).
     \end{align*}
    Since $F_n \to F$ pointwise and $|F_n - F| \leq 2G$, we can apply the Lebesgue dominated convergence theorem twice. First, we apply it to the inner cone integral in~\eqref{eq:AF_n}, to deduce that $\mathcal{A}(F_n-F)(\mathbf{x}) \to 0$ almost everywhere. Then, we apply it to the outer integral over $\mathbb{R}^N$, using the upper majorant $\mathcal{A}(2G)^p \in L^1(dw)$, to obtain
\begin{equation}\label{eq:Cauchy_almost_final}
    \lim_{n \to \infty}\|F_n-F\|_{T_2^p}=0,
\end{equation}
provided that the sequence satisfies the rapid decay condition~\eqref{eq:Cauchy_ver}. Consequently, 
    \begin{align*}
        \|F\|_{T_2^p}=\lim_{n \to \infty}\|Q_\cdot*f_n\|_{T_2^p}=\lim_{n \to \infty}\|f_n\|_{\mathbb{H}^p_{\mathrm{Dunkl}}}.
    \end{align*}
\end{proof}

\subsection{Proof of  atomic characterizations of $H^p_{\mathrm{Dunkl}}$ (Theorem~\ref{teo:H_p_coincides})}\label{sec:atomic_proof}

\begin{proof}[Proof of the inclusion $H^p_{\mathrm{Dunkl}}\subseteq H^p_{M\mathrm{-atom}}$ of Theorem~\ref{teo:H_p_coincides}]

Let $f \in H^p_{\mathrm{Dunkl}}$. Since $H^p_{\mathrm{Dunkl}}$ is defined as the completion of $\mathbb{H}^p_{\mathrm{Dunkl}}$, there exists a Cauchy sequence $(f_n)_{n\in\mathbb N}$ in $\mathbb{H}^p_{\mathrm{Dunkl}}$ that represents $f$. By Corollary~\ref{coro:distribution}, we can uniquely identify $f$ with the distributional limit of $(f_n)_{n\in\mathbb N}$ in $\mathcal S'(\mathbb{R}^N)$. Let $\mathbf M_0=\lim_{n \to \infty} \|f_n\|_{\mathbb{H}^p_{\mathrm{Dunkl}}}=\| Sf\|_{L^p(dw)}$ (see Proposition~\ref{propo:not_zero_limit}). We may assume that $\mathbf M_0>0$. Let $(f_{n_\ell})_{\ell \in \mathbb{N}}$ be a subsequence of $(f_n)_{n \in \mathbb{N}}$ such that $\|f_{n_1}\|_{\mathbb{H}^p_{\mathrm{Dunkl}}}\leq  2 \mathbf M_0$ and 
$$\|f_{n_\ell}-f_{n_{\ell+1}}\|_{\mathbb{H}^p_{\mathrm{Dunkl}}}^p\leq 2^{-\ell}\mathbf M_0^p .$$ 
Since $f_{n_{\ell+1}}-f_{n_{\ell}}\in\mathbb{H}^p_{\mathrm{Dunkl}}$, by Proposition~\ref{propo:atomic_decomposition} there exist scalars $\lambda_{j,\ell}\in\mathbb C$ and $(p,2,M,\Delta_k)$-atoms $\boldsymbol{a}_{j,\ell}$ such that
\begin{equation*}
 f_{n_{\ell+1}}-f_{n_{\ell}}=\sum_{j=1}^{\infty} \lambda_{j,\ell}\boldsymbol{a}_{j,\ell},   
\end{equation*}
where convergence is in $\mathcal{S}'(\mathbb{R}^N)$ and $L^2(dw)$, and
\begin{equation*}
    \sum_{j=1}^{\infty} |\lambda_{j,\ell} |^p\leq C\|f_{n_{\ell+1}}-f_{n_{\ell}}\|_{\mathbb{H}^p_{\mathrm{Dunkl}}}^p\leq C 2^{-\ell} \mathbf M_0^p.
\end{equation*}
 The function $f_{n_1}$ admits the atomic decomposition as well, that is, 
 $$ f_{n_1}=\sum_{j=1}^{\infty} \lambda_j\boldsymbol{a}_j, \quad \sum_{j=1}^{\infty}|\lambda_j|^p\leq C\|f_{n_1}\|_{\mathbb{H}^p_{\mathrm{Dunkl}}}^p\leq C (2\mathbf M_0)^p.$$
Consequently, 
\begin{equation}\label{eq:dist_into_atoms}
    f=\sum_{j=1}^{\infty} \lambda_j \boldsymbol{a}_j +\sum_{\ell=1}^\infty \sum_{j=1}^{\infty} \lambda_{j,\ell} \boldsymbol{a}_{j,\ell} 
\end{equation}
with
\begin{equation}\label{eq:lambdas_bound}
    \sum_{j=1}^{\infty}|\lambda_j|^p+\sum_{\ell=1}^\infty \sum_{j=1}^{\infty} |\lambda_{j,\ell}|^p\leq C'\mathbf M_0^p.
\end{equation}

By~\eqref{eq:lambdas_bound} and Corollary~\ref{distr-bound}, the series~\eqref{eq:dist_into_atoms} converges in $\mathcal{S}'(\mathbb{R}^N)$. This proves  
\[\|f\|_{H^p_{M\mathrm{-atom}}}^p \leq C'\mathbf M_0^p=C'\|f\|_{H^p_{\mathrm{Dunkl}}}^p.\]

\end{proof}

\begin{proof}[Proof of the inclusion $H^p_{M\mathrm{-atom}}\subseteq H^p_{\mathrm{Dunkl}}$ of Theorem~\ref{teo:H_p_coincides}]

{Fix $0<p\leq 1$ and  an integer $M>\mathbf N(2-p)/4p$. 
Let  $f\in \mathcal S'(\mathbb{R}^N)$ be such that 
\begin{equation}\label{eq:form_of_distribution}
    f=\sum_{j=1}^{\infty} \lambda_j \boldsymbol{a}_j, \quad \sum_{j=1}^{\infty} |\lambda_j|^p<\infty
\end{equation}
where $\boldsymbol a_j$ are $(p,2,M,\Delta_k)$-atoms. 
Set $f_n=\sum_{j=1}^n \lambda_j \boldsymbol{a}_j$.}
Then, by virtue of Proposition \ref{atom_in_Hp}, $f_n \in \mathbb{H}_{\mathrm{Dunkl}}^p$ and 
\begin{equation}
    \|f_n\|_{\mathbb{H}^p_{\mathrm{Dunkl}}}^p\leq C \sum_{j=1}^n |\lambda_j|^p \leq C \sum_{j=1}^\infty |\lambda_j|^p,
\end{equation}
\begin{equation}
     \|f_n-f_m\|_{\mathbb{H}^p_{\mathrm{Dunkl}}}^p= \|\sum_{j=n+1}^m  \lambda_j \boldsymbol{a}_j\|_{\mathbb{H}^p_{\mathrm{Dunkl}}}^p\leq \sum_{j=n+1}^m |\lambda_j|^p \|\boldsymbol{a}_j\|^p_{\mathbb{H}^p_{\mathrm{Dunkl}}}  \leq C \sum_{j=n+1}^m |\lambda_j|^p.
\end{equation}
Therefore, by~\eqref{eq:form_of_distribution}, the partial sums $f_n$ form a Cauchy sequence in $\mathbb{H}^p_{\mathrm{Dunkl}}$. Since the elements of $H^p_{\mathrm{Dunkl}}$ are uniquely identified with distributions (see Proposition~\ref{propo:not_zero_limit}), the Cauchy sequence $(f_n)_{n \in \mathbb{N}}$ represents the original distribution $f$. Consequently, passing to the limit yields the desired bound $\|f\|_{H^p_{\mathrm{Dunkl}}}^p \leq C\sum_{j=1}^\infty |\lambda_j|^p$.

\end{proof}

\section{Proof of H\"ormander's multiplier theorem}\label{sec:Hormander}

\subsection{Multipliers and estimates for the associated integral kernels}

Let $m$ be a function defined on $\mathbb{R}^N$. In this section,  we
assume that there exists $s>\frac{\mathbf{N}}{p}$ such that the multiplier $m$
satisfies~\eqref{eq:assumption}.  Let $\phi$ be a radial $C^\infty$-function on $\mathbb{R}^N$ supported in the annulus $\{\xi\in\mathbb{R}^N: \frac{1}{2}\leq \| \xi\|\leq 2  \}$ such that 
$1=\sum_{\ell \in \mathbb{Z}}\phi(2^{-\ell}\xi)$ for all $\xi \in \mathbb{R}^N \setminus \{0\}$. For $\ell \in \mathbb{Z}$ and $M>0$, we define two multipliers and their scaled versions:
\begin{equation}\label{eq:m_n}
     \widetilde{m}_\ell(\xi)=m(2^\ell \xi)\phi(\xi), \quad \widetilde{n}_\ell(\xi)=m(2^\ell\xi)\phi(\xi)2^{2M\ell}\|\xi\|^{2M},
\end{equation}
\begin{equation*}
     m_{\ell}(\xi)=m(\xi)\phi(2^{-\ell}\xi)=\widetilde{m}_{\ell}(2^{-\ell}\xi),
\end{equation*}
\begin{equation*}
     n_{\ell}(\xi)=m(\xi)\phi(2^{-\ell}\xi)\|\xi\|^{2M}=\widetilde{n}_{\ell}(2^{-\ell}\xi).
\end{equation*}
Let us denote their inverse Dunkl-Fourier transforms by
\begin{equation}\label{eq:tilde_K}
     K_\ell(\mathbf{x})=\mathcal F^{-1} ( m_\ell) (\mathbf{x}), \ \ \widetilde{K}_\ell(\mathbf{x})=\mathcal F^{-1} (\widetilde{m}_\ell) (\mathbf{x})=2^{-\mathbf{N}\ell} K_\ell (2^{-\ell}\mathbf{x}),
\end{equation}
and
\begin{equation}\label{eq:tilde_L}
     L_\ell (\mathbf{x})=\mathcal F^{-1} ( n_\ell)(\mathbf{x}), \ \ \widetilde{L}_\ell (\mathbf{x})=\mathcal F^{-1} (\widetilde{n}_\ell)(\mathbf{x})=2^{-\mathbf{N}\ell} L_\ell(2^{-\ell}\mathbf{x}).
\end{equation}
Let $\mathcal{T}_m$ and $\mathcal{T}_{\ell}$, where $\ell \in \mathbb{Z}$, denote the Dunkl multiplier operators associated with $m$ and ${m}_\ell$, respectively. 
Clearly, for $f\in L^2(dw)$ one has
\begin{equation}\label{eq:kernel_for_cutoff} 
{\mathcal T_mf=\sum_{\ell\in\mathbb Z} \mathcal T_\ell f, \quad }
\mathcal{T}_\ell f(\mathbf x)={K}_\ell*f(\mathbf x).
\end{equation}

\begin{lemma}
    Let $\delta'>\delta>0$. There is a constant $C>0$ such that for all $\ell \in \mathbb{Z}$ we have
    \begin{equation}\label{eq:Hormander_K}
        \left\|\widetilde{K}_{\ell}(\mathbf{x})(1+\|\mathbf{x}\|)^{\mathbf{N}/2+\delta}\right\|_{L^1(dw(\mathbf{x}))} \leq C\|\widetilde{m}_\ell\|_{W_2^{\mathbf{N}+\delta'}},
    \end{equation}
    \begin{equation}\label{eq:Hormander_L}
        \left\|\widetilde{L}_{\ell}(\mathbf{x})(1+\|\mathbf{x}\|)^{\mathbf{N}/2+\delta}\right\|_{L^1(dw(\mathbf{x}))} \leq C\|\widetilde{n}_\ell\|_{W_2^{\mathbf{N}+\delta'}}.
    \end{equation}
\end{lemma}

\begin{proof}
    We will prove only~\eqref{eq:Hormander_K}; the proof of~\eqref{eq:Hormander_L} follows exactly the same pattern. Let $\gamma>0$ be such that $\delta+\gamma<\delta'$. It follows from the form of the measure $dw$ (see~\eqref{eq:measure}) that $$\|(1+\|\mathbf{x}\|)^{-\mathbf{N}/2-\gamma}\|_{L^2(dw(\mathbf{x}))} \leq C_\gamma<\infty $$ for any $\gamma>0$. Therefore, by the Cauchy--Schwarz inequality,
    \begin{equation}\label{eq:K_CS}
        \left\|\widetilde{K}_{\ell}(\mathbf{x})(1+\|\mathbf{x}\|)^{\mathbf{N}/2+\delta}\right\|_{L^1(dw(\mathbf{x}))} \leq C\left\|\widetilde{K}_{\ell}(\mathbf{x})(1+\|\mathbf{x}\|)^{\mathbf{N}+\delta+\gamma}\right\|_{L^2(dw(\mathbf{x}))}. 
    \end{equation}
    Then Proposition 5.3 of~\cite{ABDH} asserts that for any real numbers $\alpha >\beta >0$ there is a constant $C=C_{\alpha,\beta}$ such that 
\begin{equation}\label{eq:insertion}
\| \widetilde{K}_\ell (\mathbf x)(1+\| \mathbf x\|)^\beta \|_{L^2(dw(\mathbf x))} \leq C \| \widetilde{m}_{\ell} \|_{W^\alpha_2} = C \| \widehat{\widetilde{m}}_{\ell}(\mathbf x) (1+\| \mathbf x\|)^{\alpha} \|_{L^2(d\mathbf x)},\end{equation}
where $\widehat{\widetilde{m}}_{\ell}$ denotes the classical Fourier transform of $\widetilde{m}_{\ell}$. Therefore,~\eqref{eq:Hormander_K} follows from~\eqref{eq:K_CS} and~\eqref{eq:insertion} applied with $\beta=\mathbf{N}+\delta+\gamma$ and $\alpha=\mathbf{N}+\delta'$.
\end{proof}

\begin{corollary}
    Let $\delta'>\delta>0$ and let $\psi$ be as in Theorem~\ref{teo:teo_main}. There is a constant $C>0$ such that for all $\ell \in \mathbb{Z}$, we have
    \begin{equation}\label{eq:Hormander_K_ver2}
        \left\|\widetilde{K}_{\ell}(\mathbf{x})(1+\|\mathbf{x}\|)^{\mathbf{N}/2+\delta}\right\|_{L^1(dw(\mathbf{x}))} \leq C \sup_{t>0}\|\psi (\cdot)m(t \cdot )\|_{W^{\mathbf{N}+\delta'}_2},
    \end{equation}
    \begin{equation}\label{eq:Hormander_L_ver2}
        \left\|\widetilde{L}_{\ell}(\mathbf{x})(1+\|\mathbf{x}\|)^{\mathbf{N}/2+\delta}\right\|_{L^1(dw(\mathbf{x}))} \leq C 2^{2M\ell} \sup_{t>0}\|\psi (\cdot)m(t \cdot )\|_{W^{\mathbf{N}+\delta'}_2}.
    \end{equation}
\end{corollary}

\begin{proof}
     One can prove that if $\eta \in C_c^\infty(\mathbb{R}^N\setminus \{0\})$ is radial, then 
     there is a constant $C_{\psi,\eta}>0$ such that
\begin{equation}\label{eq:any_function}
\sup_{t>0}\|\eta (\cdot)m(t \cdot )\|_{W^{\mathbf{N}+\delta'}_2} \leq C_{\psi,\eta} \sup_{t>0}\|\psi (\cdot)m(t \cdot )\|_{W^{\mathbf{N}+\delta'}_2}
\end{equation}
     (see Remark 1.4 of~\cite{DzHe-JFA}; note that the claim is a classical result, independent of the Dunkl setting). Therefore,~\eqref{eq:Hormander_K_ver2} and~\eqref{eq:Hormander_L_ver2} are consequences of~\eqref{eq:Hormander_K} and~\eqref{eq:Hormander_L} (and the definition of $\widetilde{m}_\ell$ and $\widetilde{n}_{\ell}$, see~\eqref{eq:m_n}) if we apply~\eqref{eq:any_function} with $\eta(\xi)=\phi(\xi)$ and $\eta(\xi)=\phi(\xi)\|\xi\|^{2M}$ respectively.
\end{proof}

\begin{corollary}
    Let $\delta'>\delta>0$ and let $\psi$ be as in Theorem~\ref{teo:teo_main}. There is a constant $C>0$ such that for all $\ell \in \mathbb{Z}$ and $R>0$, we have
    \begin{equation}\label{eq:K_1}
        \|{K}_\ell\|_{L^1(dw)} \leq C\sup_{t>0}\|\psi(\cdot)m(t\cdot)\|_{W^{\mathbf{N}+\delta'}_2},
    \end{equation}
    \begin{equation}\label{eq:K_2}
        \left\|{K}_\ell(\mathbf{x})\|\mathbf{x}\|^{\mathbf{N}/2+\delta}\right\|_{L^1(dw)} \leq C2^{-\ell(\mathbf{N}/2+\delta)}\sup_{t>0}\|\psi(\cdot)m(t\cdot)\|_{W^{\mathbf{N}+\delta'}_2},
    \end{equation}
    \begin{equation}\label{eq:K_3}
        \int_{\|\mathbf{x}\|>R}|{K}_\ell(\mathbf{x})|\,dw(\mathbf{x}) \leq CR^{-\mathbf{N}/2-\delta}2^{-\ell(\mathbf{N}/2+\delta)}\sup_{t>0}\|\psi(\cdot)m(t\cdot)\|_{W^{\mathbf{N}+\delta'}_2},
    \end{equation}
     \begin{equation}\label{eq:L_1}
        \|{L}_\ell\|_{L^1(dw)} \leq C2^{2M\ell}\sup_{t>0}\|\psi(\cdot)m(t\cdot)\|_{W^{\mathbf{N}+\delta'}_2},
    \end{equation}
    \begin{equation}\label{eq:L_2}
        \left\|{L}_\ell(\mathbf{x})\|\mathbf{x}\|^{\mathbf{N}/2+\delta}\right\|_{L^1(dw)} \leq C2^{2M\ell}2^{-\ell(\mathbf{N}/2+\delta)}\sup_{t>0}\|\psi(\cdot)m(t\cdot)\|_{W^{\mathbf{N}+\delta'}_2},
    \end{equation}
    \begin{equation}\label{eq:L_3}
        \int_{\|\mathbf{x}\|>R}|{L}_\ell(\mathbf{x})|\,dw(\mathbf{x}) \leq C2^{2M\ell}R^{-\mathbf{N}/2-\delta}2^{-\ell(\mathbf{N}/2+\delta)}\sup_{t>0}\|\psi(\cdot)m(t\cdot)\|_{W^{\mathbf{N}+\delta'}_2}.
    \end{equation}
\end{corollary}

\begin{proof}
    It is enough to prove~\eqref{eq:K_1},~\eqref{eq:K_2}, and~\eqref{eq:K_3}; the proof of~\eqref{eq:L_1},~\eqref{eq:L_2}, and~\eqref{eq:L_3} is the same (it is just a matter of an extra factor $2^{2M\ell}$, which distinguishes~\eqref{eq:Hormander_K_ver2} from~\eqref{eq:Hormander_L_ver2}). Now, let us note that~\eqref{eq:K_1} and~\eqref{eq:K_2} are simple consequences of scaling (see the definitions of $\widetilde{K}_\ell$ and ${K}_\ell$) and~\eqref{eq:Hormander_K_ver2}. In order to justify~\eqref{eq:K_3}, observe that
    \begin{equation*}
        \int_{\|\mathbf{x}\|>R}|{K}_\ell(\mathbf{x})|\,dw(\mathbf{x}) \leq  \int_{\|\mathbf{x}\|>R}|{K}_\ell(\mathbf{x})|\frac{\|\mathbf{x}\|^{\mathbf{N}/2+\delta}}{R^{\mathbf{N}/2+\delta}}\,dw(\mathbf{x}) \leq R^{-\mathbf{N}/2-\delta}\left\|{K}_\ell(\mathbf{x})\|\mathbf{x}\|^{\mathbf{N}/2+\delta}\right\|_{L^1(dw)},
    \end{equation*}
    so~\eqref{eq:K_3} follows from~\eqref{eq:K_2}.
\end{proof}

\subsection{Proof of H\"ormander's multiplier theorem}
\begin{proof}[Proof of Theorem~\ref{teo:teo_main}] {Let $0<p\leq 1$.  Suppose that $m$ satisfies \eqref{eq:assumption} with a certain $s>\mathbf N/p$.  Let $\delta>0$ be such that $\mathbf N/p<\mathbf N+\delta<s$. Fix an integer $M$ such that  $M>\mathbf N(2-p)/4p$ and  $2M>\mathbf N/2+\delta$. 
    Consider  a $(p,2,2M,\Delta_k)$-atom $\boldsymbol a$. Then there exist a ball $B=B(\mathbf{x}_0,r)$ and a function $\boldsymbol b\in\mathcal D(\Delta_k^{2M})$ such that 
    \begin{equation}
        \label{eq:support2M} \text{supp}\, \boldsymbol b\subseteq \mathcal O(B),
    \end{equation}
    \begin{equation}
        \label{eq:size2M} \|(r^2\Delta_k)^n \boldsymbol b\|_{L^2(dw)}\leq r^{4M} w(B)^{\frac{1}{2}-\frac{1}{p}}, \quad n=0,1,\dots,2M. 
    \end{equation}We will prove that there is a constant $C>0$ independent of $\boldsymbol{a}(\cdot)$ such that 
    $(M_0 C)^{-1}\mathcal{T}_m\boldsymbol{a}$ is a $(p,2,M,\Delta_k,\varepsilon)$-molecule associated with $B$, where }
    \begin{equation*}
        \varepsilon=\frac{1}{2}\Big(\delta-\mathbf{N}\Big(\frac{1}{p}-1\Big)\Big)>0.
    \end{equation*}
    For this purpose, we have to verify two conditions~\eqref{numitem:1}  and~\eqref{numitem:2}  of Definition~\ref{def:molecule}. 
    
    First, let us verify~\eqref{numitem:1}. 
    {Since $m$ is a bounded function, $\boldsymbol b':=\mathcal T_m\Delta_k^M \boldsymbol b\in \mathcal D(\Delta_k^M)$ (see~\eqref{eq:domain}) and } 
    \begin{equation*}
        \mathcal{T}_m\boldsymbol{a}=\mathcal{T}_m (\Delta_k^{2M}\boldsymbol{b})=\Delta_k^{M}(\mathcal{T}_m (\Delta_k^{M}\boldsymbol{b}))=\Delta_k^{M}\boldsymbol{b}'.
    \end{equation*}
    Therefore, $\mathcal{T}_m\boldsymbol{a}$ is of the form $\Delta_k^{M}\boldsymbol{b}'$, which proves~\eqref{numitem:1}.

    We now check~\eqref{numitem:2}. For this purpose, we have to verify the condition~\eqref{eq:molecule_condition} for $\boldsymbol{b}'$; that is, we have to find a constant $C>0$ independent of $\boldsymbol{a}$ such that for all $n=0,1,\ldots,M$ and $j =0,1,2,\ldots$ we have
    \begin{equation}\label{eq:to_verify}
        \|(r^2\Delta_k)^n(\boldsymbol{b}')\|_{L^2(U_j(B))}=\|(r^2\Delta_k)^n(\mathcal{T}_m(\Delta_k^M \boldsymbol{b}))\|_{L^2(U_j(B))} \leq CM_0 2^{-\varepsilon j} r^{2M} w(2^j B)^{\frac{1}{2}-\frac{1}{p}}.
    \end{equation}
    First, let us prove~\eqref{eq:to_verify} for $0\leq j\leq 4$. Using the fact that $\mathcal{T}_m$ is bounded on $L^2(dw)$, and, as a multiplier operator, it commutes with $\Delta_k$, we get
    \begin{align*}
       \|(r^2\Delta_k)^n\boldsymbol b'\|_{L^2(U_j(B))}&=   \|(r^2\Delta_k)^n(\mathcal{T}_m(\Delta_k^M \boldsymbol{b}))\|_{L^2(U_j(B))} =\|\mathcal{T}_m((r^2\Delta_k)^n(\Delta_k^M \boldsymbol{b}))\|_{L^2(U_j(B))} \\&\leq \|\mathcal{T}_m\|_{L^2(dw) \to L^2(dw)}\|(r^2\Delta_k)^n(\Delta_k^M \boldsymbol{b})\|_{L^2(dw)}\leq M_0 r^{2M}w(B)^{\frac{1}{2}-\frac{1}{p}}.
    \end{align*}
   {Let $ K_\ell,  L_\ell\in L^1(dw)\cap L^2(dw)$ be as in \eqref{eq:tilde_K}, \eqref{eq:tilde_L}.} In order to prove~\eqref{eq:to_verify} for $j\geq 5$, we split it as follows: 
    \begin{align*}
        (r^2\Delta_k)^n\boldsymbol b'&=(r^2\Delta_k)^n(\mathcal{T}_m(\Delta_k^M \boldsymbol{b}))=\sum_{\ell \in \mathbb{Z}}(r^2\Delta_k)^n(\mathcal{T}_\ell(\Delta_k^M \boldsymbol{b}))\\&=\sum_{2^\ell r \geq 1}{K}_\ell *((r^2 \Delta_k)^n\Delta_k^M \boldsymbol{b})+\sum_{2^\ell r<1}(\Delta_k^M {K}_\ell)*((r^2 \Delta_k)^n \boldsymbol{b}).
    \end{align*}
    We will estimate each term separately.

    \noindent\textbf{Case 1:} $2^{\ell}r \geq 1$. Let us split
    \begin{equation*}
        {K}_\ell(\mathbf{x})={K}_\ell(\mathbf{x}) \chi_{[0,2^jr]}(\|\mathbf{x}\|)+{K}_\ell(\mathbf{x}) \chi_{(2^jr,\infty)}(\|\mathbf{x}\|)={K}_\ell'(\mathbf{x})+{K}''_\ell(\mathbf{x}).
    \end{equation*}
    {By \eqref{eq:support2M}, $(r^2 \Delta_k)^n\Delta_k^M \boldsymbol{b}$ is also supported in $\mathcal{O}(B)$.  Let $ K'_\ell (\mathbf x,\mathbf y)=\tau_\mathbf x K'_\ell (-\mathbf y)$. From  Theorem~\ref{teo:support} we conclude that  $ K'_\ell (\mathbf x,\mathbf y)=0$ for $d(\mathbf x,\mathbf y)>2^jr$.  Since 
    \begin{equation*}
    \begin{split}
        {K}_\ell'*((r^2 \Delta_k)^n\Delta_k^M \boldsymbol{b})(\mathbf{x})&=\int_{\mathbb{R}^N}((r^2 \Delta_k)^n\Delta_k^M \boldsymbol{b})(\mathbf{y}) {K}_\ell'(\mathbf x,\mathbf y)\mathbf{y})\,dw(\mathbf{y})
        \end{split}
    \end{equation*}
    (see~\eqref{eq:conv_translation}), $ {K}_\ell'*((r^2 \Delta_k)^n\Delta_k^M \boldsymbol{b})$ is supported in $\mathcal{O}(2^{j+3}B)$, so $ {K}_\ell'*((r^2 \Delta_k)^n\Delta_k^M \boldsymbol{b})(\mathbf{x})=0$ for $\mathbf{x} \in U_j(B)$ and  $j\geq 5$.} Therefore,
    \begin{align}\label{eq:K=K''}
        \|{K}_\ell *((r^2 \Delta_k)^n\Delta_k^M \boldsymbol{b})\|_{L^2(U_j(B))}=\|{K}_\ell'' *((r^2 \Delta_k)^n\Delta_k^M \boldsymbol{b})\|_{L^2(U_j(B))}.
    \end{align}
    Recall that $(r^2 \Delta_k)^n\Delta_k^M \boldsymbol{b}\in L^2(dw)$ and ${K}_\ell'' \in L^1(dw) \cap L^2(dw)$ (see~\eqref{eq:K_1}). Therefore, due to Lemma~\ref{lem:ThangaveluXu}, we obtain
    \begin{align}\label{eq:K''onUj}
        \|{K}_\ell'' *((r^2 \Delta_k)^n\Delta_k^M \boldsymbol{b})\|_{L^2(U_j(B))} \leq \|(r^2 \Delta_k)^n\Delta_k^M \boldsymbol{b}\|_{L^2(dw)}\|{K}_\ell''\|_{L^1(dw)}.
    \end{align}
    From~\eqref{eq:K_3} we get
    \begin{equation}\label{K''L1}
        \|{K}_\ell''\|_{L^1(dw)} \leq CM_0r^{-(\mathbf{N}/2+\delta)} 2^{-j(\mathbf{N}/2+\delta)}2^{-\ell(\mathbf{N}/2+\delta)}.
    \end{equation}
    {The relations \eqref{eq:size2M} give 
    \begin{equation}\label{eq:2M_is_M}
        \|(r^2 \Delta_k)^n\Delta_k^M \boldsymbol{b}\|_{L^2(dw)} \leq r^{2M}w(B)^{\frac{1}{2}-\frac{1}{p}}, \quad n=0,1,\dots M.
    \end{equation}
    Consequently, from \eqref{eq:K=K''}--\eqref{eq:2M_is_M}, we conclude that  }
    \begin{equation}\label{eq:only_j}
    \begin{split}
        \sum_{2^\ell r \geq 1} & \left\|{K}_\ell *((r^2 \Delta_k)^n\Delta_k^M \boldsymbol{b})\right\|_{L^2(U_j(B))} \\
        &\leq CM_0r^{-(\mathbf{N}/2+\delta)}r^{2M}w(B)^{\frac{1}{2}-\frac{1}{p}}2^{-j(\mathbf{N}/2+\delta)}\sum_{2^{\ell}r\geq 1}2^{-\ell(\mathbf{N}/2+\delta)}\\& \leq CM_0r^{2M}w(B)^{\frac{1}{2}-\frac{1}{p}}2^{-j(\mathbf{N}/2+\delta)}.
        \end{split}
    \end{equation}
    By~\eqref{eq:growth},
    \begin{align}\begin{split}\label{eq:BallToBall}
        w(B)^{\frac{1}{2}- \frac{1}{p}}2^{-j(\frac{\mathbf{N}}{2}+\delta)}&=\Big(\frac{w(2^jB)}{w(B)}\Big)^{\frac{1}{p}-\frac{1}{2}} w(2^jB)^{\frac{1}{2}-\frac{1}{p}}2^{-j(\frac{\mathbf{N}}{2}+\delta)}\\
        &\leq C2^{j\mathbf{N}(\frac{1}{p}-\frac{1}{2})} w(2^jB)^{\frac{1}{2}-\frac{1}{p}} 2^{-j(\frac{\mathbf{N}}{2}+\delta)}.
   \end{split} \end{align}
    Finally, by our choice of $\varepsilon$ we have $0<\varepsilon < \delta-\mathbf{N}\left(\frac{1}{p}-1\right)$. Therefore, using~\eqref{eq:only_j} and~\eqref{eq:BallToBall}, we arrive at   
    \begin{equation*}
        \sum_{2^\ell r \geq 1}\left\|{K}_\ell *((r^2 \Delta_k)^n\Delta_k^M \boldsymbol{b})\right\|_{L^2(U_j(B))} \leq CM_0 2^{-\varepsilon j}w(2^jB)^{\frac{1}{2}-\frac{1}{p}},
    \end{equation*}
    which is the same bound as in~\eqref{eq:to_verify}.

     \noindent\textbf{Case 2:} $2^{\ell}r < 1$. It follows from~\eqref{eq:Laplacian_on_Fourier_side} and the definitions of $\widetilde{K}_\ell,{K}_\ell,\widetilde{L}_\ell,{L}_\ell$ that 
     \begin{equation*}
         (\Delta_k^M {K}_\ell)*((r^2 \Delta_k)^n \boldsymbol{b})=(-1)^{M}({L}_\ell)*((r^2 \Delta_k)^n \boldsymbol{b}).
     \end{equation*}
     Since $\supp (r^2 \Delta_k)^n \boldsymbol{b}\subseteq \mathcal{O}(B)$, proceeding in the same way as in \noindent\textbf{Case 1}, we have the following bound:
     \begin{equation}\label{eq:L_ell1}
         \|{L}_\ell *((r^2 \Delta_k)^n \boldsymbol{b})\|_{L^2(U_j(B))} \leq \|(r^2 \Delta_k)^n \boldsymbol{b}\|_{L^2(dw)}\|{L}_\ell(\cdot) \chi_{[2^j r,\infty)}(\|\cdot\|) \|_{L^1(dw)}.
     \end{equation}
     Next, by~\eqref{eq:L_3}, 
      \begin{equation}\label{eq:L_ell2}
        \|{L}_\ell (\cdot) \chi_{[2^j r,\infty)}(\|\cdot\|)\|_{L^1(dw)} \leq CM_0r^{-(\mathbf{N}/2+\delta)} 2^{2M\ell}2^{-j(\mathbf{N}/2+\delta)}2^{-\ell(\mathbf{N}/2+\delta)}.
    \end{equation}
   
    Consequently, applying \eqref{eq:L_ell1}, \eqref{eq:size2M}, and \eqref{eq:L_ell2}, we obtain 
    \begin{align*}
        \sum_{2^\ell r < 1}\left\|{L}_\ell *((r^2 \Delta_k)^n \boldsymbol{b})\right\|_{L^2(U_j(B))} &\leq CM_0r^{-(\mathbf{N}/2+\delta)}r^{4M}w(B)^{\frac{1}{2}-\frac{1}{p}}2^{-j(\mathbf{N}/2+\delta)}\sum_{2^{\ell}r< 1} 2^{2M\ell}2^{-\ell(\mathbf{N}/2+\delta)}.
    \end{align*}
    Since $2M>\frac{\mathbf{N}}{2}+\delta$,  we have
    \begin{equation*}
         CM_0r^{-(\mathbf{N}/2+\delta)}r^{4M}w(B)^{\frac{1}{2}-\frac{1}{p}}2^{-j(\mathbf{N}/2+\delta)}\sum_{2^{\ell}r< 1} 2^{2M\ell}2^{-\ell(\mathbf{N}/2+\delta)} \leq CM_0r^{2M}w(B)^{\frac{1}{2}-\frac{1}{p}}2^{-j(\mathbf{N}/2+\delta)},
    \end{equation*}
    which  is exactly the same bound as in~\eqref{eq:only_j}.  Therefore, using \eqref{eq:BallToBall}, we get 
    \begin{equation*}
         \sum_{2^\ell r < 1}\left\|{L}_\ell *((r^2 \Delta_k)^n \boldsymbol{b})\right\|_{L^2(U_j(B))} \leq CM_0 2^{-\varepsilon j}w(2^jB)^{\frac{1}{2}-\frac{1}{p}},
    \end{equation*}
    which finishes the proof of~\eqref{eq:to_verify} for $j\geq 5$. 

    In summary, {we have proved that there is a constant $C>0$ such that for any $(p,2,2M,\Delta_k)$-atom $\boldsymbol a$, the function $(CM_0)^{-1} \mathcal T_m \boldsymbol a$ is a $(p,2,M,\Delta_k,\varepsilon)$-molecule. Consequently, by virtue of Proposition \ref{prop:m_in_Hp},  we have }
    \begin{equation}\label{eq:on_atom}
    \|\mathcal{T}_m\boldsymbol{a}\|_{\mathbb{H}^p_{\mathrm{Dunkl}}} \leq CM_0.
    \end{equation}
    Now, let us deduce from~\eqref{eq:on_atom} that $\mathcal{T}_m$ has a unique extension to a bounded operator on $H^p_{\mathrm{Dunkl}}$. First, let us assume that $f \in \mathbb{H}^p_{\mathrm{Dunkl}}$. Then $f$ admits an atomic decomposition $f=\sum_{j=1}^{\infty}\lambda_j \boldsymbol{a}_j$,
     where the $\boldsymbol{a}_j$ are $(p,2,2M,\Delta_k)$-atoms,
    \begin{equation*}
        \Big(\sum_{j=1}^{\infty}|\lambda_j|^p\Big)^{1/p} \leq C\|f\|_{\mathbb{H}_{\mathrm{Dunkl}}^p},
    \end{equation*}
    and the series $\sum_{j=1}^{\infty}\lambda_j \boldsymbol{a}_j$ converges in $L^2(dw)$. Since $\mathcal{T}_m$ is a bounded operator on $L^2(dw)$, we have
    \begin{equation}
        \mathcal{T}_mf=\mathcal{T}_m\Big(\sum_{j=1}^{\infty}\lambda_j \boldsymbol{a}_j\Big) =\sum_{j=1}^{\infty}\lambda_j \mathcal{T}_m \boldsymbol{a}_j,
    \end{equation}
    where the series converges in $L^2(dw)$. Moreover, for all $\mathbf{x} \in \mathbb{R}^N$ we have
    \begin{equation*}
        S(\mathcal{T}_m f)(\mathbf{x}) \leq \sum_{j=1}^{\infty}|\lambda_j|S(\mathcal{T}_m\boldsymbol{a}_j)(\mathbf{x})
    \end{equation*}
    (let us recall that $S$ is defined in~\eqref{eq:square_conic}).
    
    Then, for $0<p \leq 1$ we have
    \begin{equation*}
        \left\|S(\mathcal{T}_m f)\right\|_{L^p(dw)}^p \leq \sum_{j=1}^{\infty}|\lambda_j|^p \|S(\mathcal{T}_m \boldsymbol{a}_j)\|_{L^p(dw)}^p \leq CM_0^p\sum_{j=1}^{\infty}|\lambda_j|^p \leq CM_0^{p}\|f\|_{\mathbb{H}^p_{\mathrm{Dunkl}}}^p.
    \end{equation*}
    Therefore, $\mathcal{T}_m$ is bounded on $\mathbb{H}_{\mathrm{Dunkl}}^p$, which is a dense subspace of $H^p_{\mathrm{Dunkl}}$. Consequently, it has a unique extension to a bounded operator on $H_{\mathrm{Dunkl}}^p$.
\end{proof}

\section{Appendix -- proof of Lemma \ref{lem:L_2_Schwartz}}

\begin{lemma}\label{lem:ord_to_Dunkl}
    Let $\gamma \geq 0$ and let $\ell$ be a non-negative integer. There exists a constant $C>0$, depending only on $(\mathcal{R},k,N)$, $\gamma$, and $\ell$, such that if $\varphi \in C^{\ell}(\mathbb{R}^N)$ satisfies
    \begin{equation}\label{eq:induction_ord}
        |\partial^{{\boldsymbol{\beta}}}\varphi(\mathbf{x})| \leq (1+\|\mathbf{x}\|)^{-\mathbf{N}/2-\gamma} \text{ for all } {\boldsymbol{\beta}} \in \mathbb{N}_0^N \text{ with } |{\boldsymbol{\beta}}| \leq \ell \text{ and }\mathbf{x} \in \mathbb{R}^N,
    \end{equation}
    then
    \begin{equation}\label{eq:induction_dunkl}
        |T^{{\boldsymbol{\beta}}}\varphi(\mathbf{x})| \leq C(1+\|\mathbf{x}\|)^{-\mathbf{N}/2-\gamma} \text{ for all } {\boldsymbol{\beta}} \in \mathbb{N}_0^N \text{ with } |{\boldsymbol{\beta}}| \leq \ell \text{ and }\mathbf{x} \in \mathbb{R}^N.
    \end{equation}
\end{lemma}

\begin{proof}
    The proof closely follows the argument in~\cite[pages 284--285]{DzH-square}, with additional control of the constants. For completeness and the reader’s convenience, we present the proof here. By the definition of $T_j$ and by the fundamental theorem of calculus, for all $f \in C^{1}(\mathbb{R}^N)$, we have
\begin{equation*}\begin{split} T_jf(\mathbf x)&
=\partial_j f(\mathbf x)-\sum_{\alpha\in \mathcal{R}} \frac{k(\alpha)}{2}\alpha_j\langle \mathbf{x}, \alpha \rangle^{-1}\int_0^{1} \frac{d}{dt}(f(\mathbf{x}-2t\alpha \|\alpha\|^{-2}\langle \mathbf{x}, \alpha\rangle ))\,dt\\
&=\partial_j f(\mathbf x)+ \sum_{\alpha\in \mathcal{R}}\frac{k(\alpha)}{2}\alpha_j \int_0^1 \langle {(}\nabla_{\mathbf{x}} f{)}(\mathbf{x}-2t\alpha \|\alpha\|^{-2}\langle \mathbf{x}, \alpha\rangle ), \alpha \rangle \,dt
\end{split}\end{equation*}
(cf.~\cite[page 9]{Roesler-Voit}). Consequently, for any ${\boldsymbol{\beta}} \in \mathbb{N}_0^N$ there is a constant $C>0$, depending only on $(\mathcal{R},k,N)$ and ${\boldsymbol{\beta}}$, such that for all $f \in C^{|{\boldsymbol{\beta}}|+1}(\mathbb{R}^N)$  and $j \in \{1,\ldots,N\}$ one has 
\begin{equation}\label{eq:s_x_1}
\sup_{\mathbf x\in\mathbb{R}^N} |\partial^{\boldsymbol{\beta}} T_j f(\mathbf x)|\leq C \sup_{\mathbf x\in\mathbb{R}^N} \|\nabla_{\mathbf{x}}\partial^{{\boldsymbol{\beta}}}f(\mathbf x)\|.
\end{equation}
It follows from~\eqref{eq:s_x_1} by induction  on $|{\boldsymbol{\beta}}|$ that for any ${\boldsymbol{\beta}} \in \mathbb{N}_0^N$ there is a constant $C_{{\boldsymbol{\beta}}}=C_{{\boldsymbol{\beta}},\mathcal{R},k,N}>0$ such that for all $f \in C^{|{\boldsymbol{\beta}}|}(\mathbb{R}^N)$ we have
\begin{equation}\label{eq:small_x}
\|T^{{\boldsymbol{\beta}}}f\|_{L^{\infty}} \leq C_{{\boldsymbol{\beta}}}\sum_{{\boldsymbol{\beta}}' \in \mathbb{N}_0^{N},\, |{\boldsymbol{\beta}}'| = |{\boldsymbol{\beta}}|}\|\partial^{{\boldsymbol{\beta}}'}f\|_{L^{\infty}}.
\end{equation}
The proof of~\eqref{eq:induction_dunkl} is by induction on $\ell$. The claim for $\ell=0$ is obvious. Assume that~\eqref{eq:induction_ord} implies~\eqref{eq:induction_dunkl} for ${0}\leq \ell \leq \ell_1$. We will prove that~\eqref{eq:induction_ord}  implies~\eqref{eq:induction_dunkl} for $\ell=\ell_1+1$.  By~\eqref{eq:small_x}, it is enough to prove~\eqref{eq:induction_dunkl} for $\|\mathbf{x}\|>2$.
 Let $\alpha \in \mathcal{R}$. By the definition of $T_{j}$, it is enough to show that there is a constant $C>0$, depending only on $\ell,\mathcal{R},k,N$, such that the function
$$\mathbf x\mapsto C^{-1}\frac{\varphi(\mathbf{x})-\varphi(\sigma_{\alpha}(\mathbf{x}))}{\langle \mathbf{x},\alpha \rangle}$$ satisfies~\eqref{eq:induction_ord} with $\ell=\ell_1$. Let ${\boldsymbol{\beta}} \in \mathbb{N}_0^{N}$ be such that $|{\boldsymbol{\beta}}| \leq \ell_1$. We consider two cases.
\\
\noindent\textbf{Case 1.} $|\langle \mathbf{x},\alpha \rangle| \geq {1/10}$. Note that there is a constant $C=C_{\ell_1}>0$ independent of $\mathbf{x}$ such that for all $|{\boldsymbol{\beta}}'| \leq \ell_1$ we have
\begin{equation}\label{eq:scalar_est}
    |\partial^{{\boldsymbol{\beta}}'}_{\mathbf{x}}(\langle \mathbf{x},\alpha \rangle^{-1})| \leq C.
\end{equation}
Due to the Leibniz rule, the estimate for $\partial^{{\boldsymbol{\beta}}}\Big(\frac{\varphi(\mathbf{x})-\varphi(\sigma_{\alpha}(\mathbf{x}))}{\langle \mathbf{x},\alpha \rangle}\Big)$ is a consequence of~\eqref{eq:induction_ord} and~\eqref{eq:scalar_est} (with the appropriate control of the constants).
\\
\noindent\textbf{Case 2.} $|\langle \mathbf{x},\alpha \rangle| < {1/10}$. We have
\begin{equation}\label{eq:diff_argument}
\begin{split}
    \frac{\varphi(\mathbf{x})-\varphi(\sigma_{\alpha}(\mathbf{x}))}{\langle \mathbf{x}, \alpha \rangle}&={-}\langle \mathbf{x}, \alpha \rangle^{-1}\int_0^{1} \frac{d}{dt}(\varphi(\mathbf{x}-2t\alpha \|\alpha\|^{-2}\langle \mathbf{x}, \alpha\rangle ))\,dt\\&= \int_0^1 \langle (\nabla_{\mathbf{x}} \varphi)(\mathbf{x}-2t\alpha \|\alpha\|^{-2}\langle \mathbf{x}, \alpha\rangle ), \alpha \rangle \,dt.
\end{split}
\end{equation}
By our assumptions, $|\langle \mathbf{x},\alpha \rangle| < {1/10}$ { and $\| \mathbf x\|>2$}, so for all $t \in [0,1]$ we have $\|\mathbf{x}-2t\alpha \|\alpha\|^{-2}\langle \mathbf{x}, \alpha\rangle\| \sim \|\mathbf{x}\|$,
so, by~\eqref{eq:induction_ord} with $\ell=\ell_1+1$, we obtain
\begin{align*}
    &\left|\partial^{{\boldsymbol{\beta}}}_{\mathbf{x}}\Big(\frac{\varphi(\mathbf{x})-\varphi(\sigma_{\alpha}(\mathbf{x}))}{\langle \mathbf{x}, \alpha \rangle}\Big)\right| \leq \left| \int_0^1 \langle \partial^{{\boldsymbol{\beta}}}_{\mathbf{x}}{\big\{}(\nabla_{\mathbf{x}} \varphi)(\mathbf{x}-2t\alpha \|\alpha\|^{-2}\langle \mathbf{x}, \alpha\rangle ){\big\}}, \alpha \rangle \,dt\right| \\&\leq C_{{\boldsymbol{\beta}}}\int_{0}^{1}\left(1+\|\mathbf{x}-2t\alpha \|\alpha\|^{-2}\langle \mathbf{x}, \alpha\rangle\|\right)^{-\mathbf{N}/2-\gamma}\,dt \leq C(1+\|\mathbf{x}\|)^{-\mathbf{N}/2-\gamma},
\end{align*}
which completes the proof of \eqref{eq:induction_dunkl} (with the appropriate control of the constants).
\end{proof}

\begin{proof}[Proof of Lemma \ref{lem:L_2_Schwartz}]
    If $\|f\|_{(\mathbf{n})}=0$, then $f \equiv 0$ and the inequality holds trivially. Assume $\|f\|_{(\mathbf{n})}>0$. Since $f \in \mathcal{S}(\mathbb{R}^N)$, for all $\mathbf{x} \in \mathbb{R}^N$ and $\boldsymbol \beta \in \mathbb{N}_0^N$ such that $|\boldsymbol \beta| \leq 2M_1$ we have
    \begin{align*}
        |\partial^{\boldsymbol \beta}f(\mathbf{x})| \leq (1+\|\mathbf{x}\|)^{-\mathbf{n}}\|f\|_{(\mathbf{n})}.
    \end{align*}
    Then, by Lemma~\ref{lem:ord_to_Dunkl} there is a constant $C>0$ independent of $f$ such that
    \begin{equation}\label{eq:transformed_to_Dunkl}
        |T^{\boldsymbol \beta}f(\mathbf{x})| \leq C(1+\|\mathbf{x}\|)^{-\mathbf{n}}\|f\|_{(\mathbf{n})}.
    \end{equation}
    The operator $(\Delta_k)^{M_1}$ is a finite linear combination of operators $T^{\boldsymbol \beta}$ of order $|\boldsymbol \beta| = 2M_1$. By the triangle inequality and estimate \eqref{eq:transformed_to_Dunkl}, there exists a constant $C>0$  such that for all $\mathbf{x} \in \mathbb{R}^N$ we have
\begin{equation}\label{eq:laplacian_bound}
    |(\Delta_k)^{M_1}f(\mathbf{x})| \leq \sum_{|\boldsymbol \beta|=2M_1} c_{\boldsymbol \beta} |T^{\boldsymbol \beta}f(\mathbf{x})| \leq C(1+\|\mathbf{x}\|)^{-\mathbf{n}}\|f\|_{(\mathbf{n})}.
\end{equation}
    Finally, applying~\eqref{eq:laplacian_bound} and the fact that $\mathbf{n}>\frac{\mathbf{N}}{2}$ (so the function $\mathbf{x} \mapsto (1+\|\mathbf{x}\|)^{-2\mathbf{n}}$ is integrable with respect to $dw$) we obtain
    \begin{align*}
        \|(\Delta_k)^{M_1}f\|_{L^2(dw)}^2&=\int_{\mathbb{R}^N}|(\Delta_k)^{M_1}f(\mathbf{x})|^2\,dw(\mathbf{x})\leq C\|f\|_{(\mathbf{n})}^2\int_{\mathbb{R}^N}(1+\|\mathbf{x}\|)^{-2\mathbf{n}}\,dw(\mathbf{x}) \leq C\|f\|_{(\mathbf{n})}^2 .   \end{align*}
    Taking the square root of both sides concludes the proof.
\end{proof}

\end{document}